\RequirePackage{fix-cm}
\documentclass[smallextended]{svjour3}

\smartqed

\usepackage[T1]{fontenc}
\usepackage{algorithm}
\usepackage{algpseudocode}
\usepackage{amsmath,amssymb}
\usepackage{bm}
\usepackage{booktabs}
\usepackage{enumitem}
\usepackage{graphicx}
\usepackage{hyperref}
\usepackage{mathtools}
\usepackage{microtype}
\usepackage{xcolor}
\usepackage{tikz}
\usepackage[caption=false]{subfig}

\hypersetup{
  colorlinks=true,
  linkcolor=blue!55!black,
  citecolor=blue!55!black,
  urlcolor=blue!55!black
}

\newcommand{\dd}{\,\mathrm{d}}

\newcommand{\Emech}{E_{\mathrm{mech}}}
\newcommand{\Etot}{E_{\mathrm{tot}}}
\newcommand{\Efield}{\boldsymbol E}
\newcommand{\J}{\boldsymbol J}
\newcommand{\uh}{\boldsymbol u_h}
\newcommand{\Bh}{\boldsymbol B_h}
\newcommand{\Jh}{\boldsymbol J_h}
\newcommand{\Eh}{\boldsymbol E_h}

\journalname{Journal of Scientific Computing}

\begin{document}

\title{Structure-preserving operator splitting for 2.5D ideal MHD with an entropy-stable DGSEM and exactly divergence-free compatible finite elements}
\titlerunning{Structure-preserving operator splitting for 2.5D ideal MHD}

\author{Siyuan Fan \and Guosheng Fu\thanks{Corresponding author: Guosheng Fu}}
\authorrunning{Fan and Fu}

\institute{S. Fan and G. Fu \at
  Department of Applied and Computational Mathematics and Statistics,
  University of Notre Dame, Notre Dame, IN 46556, USA\\
  \email{sfan23@nd.edu; gfu@nd.edu}
}

\date{}

\maketitle

\begin{abstract}
We develop a high-order, fully explicit operator-splitting method for 2.5D
ideal magnetohydrodynamics on Cartesian meshes. The hydrodynamic subflow is
advanced by an entropy-stable discontinuous Galerkin spectral element method
equipped with stagewise oscillation elimination and a positivity-preserving
limiter. The magnetic--velocity subflow uses compatible finite elements and
mass-lumped reconstructions of electric field and current density. Its
discrete-curl update exactly preserves the global $H(\mathrm{div})$
divergence-free subspace, while the ideal semidiscretization balances kinetic,
magnetic, and internal energy. Curl-form artificial resistivity and a
direction-resolved velocity filter return removed magnetic and kinetic energy
to internal energy; consequently, the stabilized magnetic stage preserves
positive internal energy and satisfies a discrete entropy inequality. The two
solvers are composed by a second-order hydrodynamic--magnetic--hydrodynamic
Strang splitting, yielding a matrix-free scheme that retains global mass,
nodal admissibility, and magnetic divergence on accepted steps. Smooth
Alfv\'en-wave and advected-vortex tests reveal an even--odd convergence pattern
in the magnetic polynomial degree, verify second-order temporal accuracy, and
show that stabilization preserves high-order accuracy. Field-loop,
Orszag--Tang, rotor, MHD blast-wave, and Kelvin--Helmholtz calculations
demonstrate robust performance for nonsmooth multidimensional flows and
show that artificial resistivity suppresses grid-scale magnetic oscillations while
retaining the resolved structures.
\end{abstract}

\section{Introduction}

Ideal magnetohydrodynamics (MHD) describes the interaction of a conducting
compressible fluid with a magnetic field~\cite{goedbloed2019magnetohydrodynamics}.
Its numerical approximation is difficult because the nonlinear wave dynamics,
thermodynamic admissibility, and magnetic geometry cannot be treated
independently. A robust high-order method should propagate shocks
conservatively, maintain positive density and pressure, control entropy
production, represent kinetic--magnetic energy exchange accurately, and
preserve the involution
\(\nabla\cdot\boldsymbol B=0\). These requirements are strongly coupled. In
particular, a nonzero discrete magnetic divergence can alter the wave structure
and jeopardize positivity, while componentwise limiting of the magnetic field
can destroy an otherwise exact divergence constraint
\cite{toth2000divergence,wu2018positivity,wu2018provably}.

Magnetic-divergence control has therefore shaped the development of MHD
solvers. Projection methods remove divergence error after an update
\cite{BRACKBILL1980426}; eight-wave formulations transport it with the flow
\cite{Powell1997}; and constrained-transport methods evolve face-normal
magnetic fluxes through a discrete curl
\cite{evans1988ct,gardiner2005unsplit}. High-order discontinuous Galerkin (DG)
methods have also been constructed with locally or globally divergence-free
magnetic spaces~\cite{Li2005LocallyDD,10.1007/s10915-018-0750-6}. These
approaches differ in whether solenoidality holds locally or globally, whether
the conservative MHD system is modified, and how the magnetic update interacts
with shock stabilization and positivity enforcement.

Entropy stability and physical admissibility provide a second, complementary
line of control. Entropy-conservative two-point fluxes combined with
summation-by-parts flux differencing form the basis of many high-order Euler and
MHD discretizations
\cite{tadmor1987numerical,tadmor2003entropy,gassner2016split,winters2016mhdentropy,derigs2016mhdsolver,chan2018discretely}.
For the Euler equations, scaling limiters enforce positive nodal states while
preserving admissible cell averages~\cite{zhang2010positivity}; for ideal MHD,
positivity analysis shows that admissibility and discrete divergence control
are intrinsically linked~\cite{wu2018positivity,wu2018provably}. Underresolved
solutions are commonly regularized through subcell indicators or artificial
viscosity~\cite{persson2006subcell,barter2010shock}. Entropy filtering and
oscillation-eliminating DG (OEDG) methods provide a complementary approach
\cite{dzanic2023mhdentropyfilter,lu2021oscillation,liu2022essentially,peng2025oedg}.
Recent MHD schemes combine entropy control, divergence-free
spaces, positivity preservation, and oscillation elimination within monolithic
DG formulations~\cite{liu2025mhdentropy,liu2025mhd,liu2026gdf,wu2026nodal}.
The resulting
methods demonstrate what can be achieved, but also expose the complexity of
enforcing all of these structures simultaneously in one discretization.

Compatible finite element methods take a geometric route to the magnetic
constraint. N\'ed\'elec and Raviart--Thomas spaces placed in a discrete de Rham
sequence make the induction update an algebraic curl and hence preserve a
globally divergence-free magnetic field. This construction has produced stable
MHD discretizations with exact magnetic-divergence control and paired energy
transfers~\cite{hu2017stable,hu2019structure}, including transport-stabilized
methods that retain compatible conservation laws~\cite{wimmer2024compatible}.
It is especially natural when current density and electric field are retained
as auxiliary variables, but combining it with a conservative, shock-robust
high-order discretization of the compressible flow remains nontrivial.

Operator splitting offers a way to separate these responsibilities. Early
split finite-volume methods isolated fluid and induction updates
\cite{fuchs2009splitting}. More recently, Dao, Nazarov, and Tomas introduced a
structure-preserving decomposition of compressible ideal MHD into a
hydrodynamic update, a magnetic--velocity update, and an energy transfer
\cite{dao2024structure}; Pang and Wu subsequently combined this framework with
positivity-preserving finite volumes and constrained transport
\cite{pang2025ct}. The present work starts from the same continuous
decomposition. We pair its hydrodynamic map with an entropy-stable
DG spectral element method (DGSEM), stagewise OEDG stabilization, and a
positivity limiter, while the magnetic--velocity map is discretized in a
compatible finite element complex. Electric field and current density are
reconstructed with diagonal mass-lumped products, and the magnetic mass matrix
is applied but never inverted. The resulting update is fully explicit and
retains the discrete curl structure without componentwise filtering of the
magnetic field.

In this work, all fields depend on \((x,y,t)\), while velocity and magnetic
field retain three components; we refer to this setting as 2.5D. The main
contributions are as follows:
\begin{enumerate}[label=(\roman*)]
\item We construct a high-order DGSEM/compatible-finite-element realization of
  the split MHD system on Cartesian meshes. The Euler variables use degree
  \(p\), the magnetic complex uses degree \(m=p\) or \(m=p-1\), and the
  collocated and mass-lumped products make the complete algorithm matrix-free.
\item We identify the properties of each spatial discretization at their
  precise level of validity. The Euler semidiscretization is conservative and
  entropy stable, and its cell averages remain admissible under
  a suitable time-step restriction. The magnetic semidiscretization preserves
  the global \(H(\mathrm{div})\) constraint and, under periodic boundary
  conditions, exactly balances kinetic, magnetic, and internal energy.
\item We introduce stagewise stabilization tailored to the two subflows. The
  Euler and velocity filters use direction-resolved cross-line OEDG scaling,
  while a curl-form artificial resistivity regularizes the magnetic field.
  Resolved kinetic and magnetic energy removed in the magnetic stage is
  returned to internal energy. Consequently, the stabilized magnetic update
  preserves positive internal energy and satisfies a discrete entropy
  inequality.
\item We combine the two solvers in a fully explicit hydrodynamic--magnetic--
  hydrodynamic Strang splitting. The spatial studies reveal a pronounced
  even--odd accuracy pattern in the compatible magnetic degree: even values
  of \(m\) give the more reliable high-order behavior. Separate studies
  verify second-order temporal accuracy and show that stabilization retains
  the observed smooth-solution orders. Robustness is demonstrated on
  challenging nonsmooth multidimensional benchmarks involving shocks,
  current sheets, and shear instabilities. Paired calculations further show
  that artificial resistivity suppresses grid-scale magnetic oscillations
  while retaining the principal resolved structures.
\end{enumerate}

The analysis is restricted to conforming Cartesian rectangular meshes. This
setting makes the auxiliary reconstructions local, supports tensor-product
mass lumping, and exposes the exact discrete curl sequence used in the
divergence proof. Curvilinear geometry, nonconforming refinement,
and full three-dimensional MHD lie outside the present scope.

The paper is organized as follows. Section~\ref{sec:model-split} states the
2.5D equations and the continuous operator split. Section~\ref{sec:numerical-method}
develops the common spatial framework, the two subflow discretizations, and
the fully discrete Strang algorithm. Section~\ref{sec:numerics} presents the
numerical experiments. Section~\ref{sec:conclusion} concludes the paper.

\section{Continuous model and operator splitting}
\label{sec:model-split}

\subsection{The 2.5D ideal-MHD equations}

Let \(\Omega\subset\mathbb R^2\) be a rectangular domain. We consider the
standard ideal-MHD equations \cite{goedbloed2019magnetohydrodynamics} under
translational invariance in the third spatial direction. Thus, all fields
depend on \((x,y,t)\), whereas the velocity and magnetic field retain all
three components:
\[
  \boldsymbol u=(u_x,u_y,u_z)^T,
  \qquad
  \boldsymbol B=(B_x,B_y,B_z)^T.
\]
This two-dimensional, three-component setting is commonly
referred to as 2.5D \cite{toth2000divergence,guillet2019highorder}.
Accordingly, the 2.5D differential operators are
\[
\begin{aligned}
  \nabla f&=(\partial_xf,\partial_yf,0)^T,\\
  \nabla\cdot\boldsymbol a&=\partial_xa_x+\partial_ya_y,\\
  \nabla\times\boldsymbol a
  &=(\partial_ya_z,-\partial_xa_z,
     \partial_xa_y-\partial_ya_x)^T.
\end{aligned}
\]
The analysis below assumes periodic boundary conditions. Boundary conditions
for the numerical experiments are stated with the corresponding test problem.

The conserved variables are the density \(\rho\), momentum
\(\boldsymbol m=\rho\boldsymbol u\), magnetic field \(\boldsymbol B\), and
total energy
\begin{equation}
\label{eq:mhd-energy-decomposition}
  \Etot
  =\rho e+\frac12\rho|\boldsymbol u|^2+\frac12|\boldsymbol B|^2
  =\Emech+\frac12|\boldsymbol B|^2,
  \qquad
  \Emech=\rho e+\frac12\rho|\boldsymbol u|^2.
\end{equation}
We use the ideal-gas equation of state
\[
  p=(\gamma-1)\rho e,
  \qquad \gamma>1,
\]
and normalize the magnetic permeability to one. The admissible set is
specified by \(\rho>0\) and \(p>0\).

In these variables, the conservative 2.5D ideal-MHD equations are
\begin{align}
  \partial_t\rho+\nabla\cdot(\rho\boldsymbol u)&=0,
  \label{eq:mhd-mass}\\
  \partial_t(\rho\boldsymbol u)
  +\nabla\cdot\left[
    \rho\boldsymbol u\otimes\boldsymbol u
    +\left(p+\frac12|\boldsymbol B|^2\right)I
    -\boldsymbol B\otimes\boldsymbol B
  \right]&=0,
  \label{eq:mhd-momentum}\\
  \partial_t\Etot
  +\nabla\cdot\left[
    \left(\Etot+p+\frac12|\boldsymbol B|^2\right)\boldsymbol u
    -(\boldsymbol u\cdot\boldsymbol B)\boldsymbol B
  \right]&=0,
  \label{eq:mhd-energy}\\
  \partial_t\boldsymbol B
  -\nabla\times(\boldsymbol u\times\boldsymbol B)&=0.
  \label{eq:mhd-induction}
\end{align}
The magnetic field is subject to
\begin{equation}
\label{eq:mhd-divergence-constraint}
  \nabla\cdot\boldsymbol B
  =\partial_xB_x+\partial_yB_y=0.
\end{equation}
Applying divergence to \eqref{eq:mhd-induction} gives
\(\partial_t(\nabla\cdot\boldsymbol B)=0\). Thus a divergence-free initial
field remains divergence-free for every smooth solution. Under the periodic
boundary assumption, Eqs.~\eqref{eq:mhd-mass}--\eqref{eq:mhd-energy} also
conserve total mass, momentum, and energy.

For positive density and pressure, divergence-free ideal MHD inherits the
Euler entropy pair. With physical specific entropy
\[
  s=\log p-\gamma\log\rho,
\]
smooth solutions satisfy
\begin{equation}
\label{eq:mhd-physical-entropy}
  \partial_t(\rho s)+\nabla\cdot(\rho s\boldsymbol u)=0.
\end{equation}
Equivalently, the mathematical entropy and entropy flux
\[
  \eta=-\frac{\rho s}{\gamma-1},
  \qquad
  \boldsymbol q=\eta\boldsymbol u
\]
satisfy \(\partial_t\eta+\nabla\cdot\boldsymbol q=0\) for smooth solutions;
entropy-admissible weak solutions satisfy the corresponding inequality
\(\partial_t\eta+\nabla\cdot\boldsymbol q\leq0\).

\subsection{Equivalent Euler--magnetic form}

With the current density and ideal electric field defined by
\begin{equation}
\label{eq:mhd-current-electric}
  \J=\nabla\times\boldsymbol B,
  \qquad
  \Efield=-\boldsymbol u\times\boldsymbol B,
\end{equation}
the Maxwell-stress identity reads
\begin{equation}
\label{eq:maxwell-stress-identity}
  \nabla\cdot\left(
    \frac12|\boldsymbol B|^2I-\boldsymbol B\otimes\boldsymbol B
  \right)
  =\boldsymbol B\times\J
   -(\nabla\cdot\boldsymbol B)\boldsymbol B.
\end{equation}
Hence, for a smooth divergence-free magnetic field, the momentum equation
separates into the Euler flux and the Lorentz force, and the conservative
system is equivalent to the following Euler--magnetic source form:
\begin{align}
  \partial_t\rho+\nabla\cdot(\rho\boldsymbol u)&=0,
  \label{eq:mhd-euler-source-mass}\\
  \partial_t(\rho\boldsymbol u)
  +\nabla\cdot(\rho\boldsymbol u\otimes\boldsymbol u+pI)
  &=-\boldsymbol B\times\J,
  \label{eq:mhd-euler-source-momentum}\\
  \partial_t\Emech
  +\nabla\cdot((\Emech+p)\boldsymbol u)
  &=-(\boldsymbol B\times\J)\cdot\boldsymbol u,
  \label{eq:mhd-euler-source-energy}\\
  \partial_t\boldsymbol B+\nabla\times\Efield&=0.
  \label{eq:mhd-euler-source-induction}
\end{align}
To verify the mechanical-energy equation, dot Faraday's law with
\(\boldsymbol B\) and use
\(\nabla\cdot(\Efield\times\boldsymbol B)
=\boldsymbol B\cdot(\nabla\times\Efield)-\Efield\cdot\J\). This gives the
magnetic-energy balance
\begin{equation}
\label{eq:mhd-magnetic-energy-balance}
  \partial_t\left(\frac12|\boldsymbol B|^2\right)
  +\nabla\cdot(\Efield\times\boldsymbol B)
  =-\Efield\cdot\J.
\end{equation}
Because
\[
  \Efield\cdot\J
  =-(\boldsymbol B\times\J)\cdot\boldsymbol u,
  \qquad
  \Efield\times\boldsymbol B
  =|\boldsymbol B|^2\boldsymbol u
   -(\boldsymbol u\cdot\boldsymbol B)\boldsymbol B,
\]
subtracting \eqref{eq:mhd-magnetic-energy-balance} from the total-energy
equation recovers \eqref{eq:mhd-euler-source-energy}. Thus the Lorentz work
in the mechanical equation is exactly the negative of the electromagnetic
work in the magnetic-energy equation.

The thermodynamic part of the model is unaffected by this exchange. Taking
the scalar product of \eqref{eq:mhd-euler-source-momentum} with
\(\boldsymbol u\) and subtracting the resulting kinetic-energy equation from
\eqref{eq:mhd-euler-source-energy} yields
\begin{equation}
\label{eq:mhd-internal-energy}
  \partial_t(\rho e)+\nabla\cdot(\rho e\boldsymbol u)
  +p\nabla\cdot\boldsymbol u=0.
\end{equation}
This is the compressible Euler internal-energy equation and leads directly to
the entropy balance \eqref{eq:mhd-physical-entropy}.

\subsection{Hydrodynamic and magnetic--velocity subflows}

Operator splittings between fluid and magnetic dynamics have earlier
precedents in the MHD literature \cite{fuchs2009splitting}. The particular
split used here---an Euler subflow
with fixed \(\boldsymbol B\) and a coupled magnetic--velocity subflow with
fixed \(\rho\) and \(e\)---was introduced by Dao et al.~\cite{dao2024structure}
and subsequently adopted by Pang and Wu~\cite{pang2025ct}.

We write the state in mechanical-energy variables as
\[
  \mathsf Q=(\rho,\boldsymbol m,\Emech,\boldsymbol B).
\]
Equations
\eqref{eq:mhd-euler-source-mass}--\eqref{eq:mhd-euler-source-induction} then
define a decomposition of the continuous vector field,
\begin{equation}
\label{eq:continuous-operator-split}
  \partial_t\mathsf Q
  =\mathcal H(\mathsf Q)+\mathcal M(\mathsf Q),
\end{equation}
where \(\mathcal H\) contains the Euler transport terms and \(\mathcal M\)
contains the Lorentz and induction terms.

The hydrodynamic subflow freezes \(\boldsymbol B\) and solves
\begin{align}
  \partial_t\rho+\nabla\cdot(\rho\boldsymbol u)&=0,
  \label{eq:hydro-continuous-mass}\\
  \partial_t(\rho\boldsymbol u)
  +\nabla\cdot(\rho\boldsymbol u\otimes\boldsymbol u+pI)&=0,
  \label{eq:hydro-continuous-momentum}\\
  \partial_t\Emech
  +\nabla\cdot((\Emech+p)\boldsymbol u)&=0,
  \label{eq:hydro-continuous-energy}\\
  \partial_t\boldsymbol B&=0.
  \label{eq:hydro-continuous-magnetic}
\end{align}
We denote its exact flow over an interval of length \(\tau\) by
\(\varphi_H^\tau\). This is the ordinary 2.5D compressible Euler system, so
the hydrodynamic stage carries the conservative transport, entropy
admissibility, and thermodynamic evolution of the split model.

The magnetic--velocity subflow freezes the density and solves
\begin{align}
  \partial_t\rho&=0,
  \label{eq:mag-rho}\\
  \partial_t(\rho\boldsymbol u)&=-\boldsymbol B\times\J,
  \label{eq:mag-momentum}\\
  \partial_t\Emech&=-(\boldsymbol B\times\J)\cdot\boldsymbol u,
  \label{eq:mag-mechanical-energy}\\
  \partial_t\boldsymbol B+\nabla\times\Efield&=0.
  \label{eq:mag-B-conserved}
\end{align}
Its exact flow is denoted by \(\varphi_M^\tau\). Since \(\rho\) is fixed,
taking the scalar product of \eqref{eq:mag-momentum} with
\(\boldsymbol u\) shows that
\[
  \partial_t\left(\frac12\rho|\boldsymbol u|^2\right)
  =-(\boldsymbol B\times\J)\cdot\boldsymbol u.
\]
Comparison with \eqref{eq:mag-mechanical-energy} gives
\(\partial_t(\rho e)=0\), and hence \(\rho\partial_t e=0\).
The magnetic subflow is therefore equivalently expressed in the primitive
variables used by the compatible stage solver:
\begin{align}
  \partial_t\rho&=0,
  &\rho\partial_t e&=0,
  \label{eq:mag-rhoe}\\
  \rho\,\partial_t\boldsymbol u+\boldsymbol B\times\J&=0,
  &\partial_t\boldsymbol B+\nabla\times\Efield&=0,
  \label{eq:mag-velocity-B}\\
  \J&=\nabla\times\boldsymbol B,
  &\Efield&=-\boldsymbol u\times\boldsymbol B.
  \label{eq:mag-constitutive}
\end{align}

The division of variables in the ideal continuous split can be summarized as
\begin{center}
\begin{tabular}{@{}lll@{}}
\toprule
Subflow & Advanced variables & Frozen variables\\
\midrule
\(\varphi_H^\tau\) & \(\rho,\boldsymbol m,\Emech\)
                    & \(\boldsymbol B\)\\
\(\varphi_M^\tau\) & \(\boldsymbol u,\boldsymbol B\)
                    & \(\rho,e\)\\
\bottomrule
\end{tabular}
\end{center}
In the hydrodynamic subflow, entropy evolves exactly as in the Euler
equations. In the ideal magnetic subflow, the thermodynamic state remains
unchanged pointwise, while kinetic and magnetic energy exchange through the
equal and opposite work terms in \eqref{eq:mag-mechanical-energy} and
\eqref{eq:mhd-magnetic-energy-balance}. Both subflows preserve
\(\nabla\cdot\boldsymbol B=0\): the hydrodynamic subflow leaves
\(\boldsymbol B\) fixed, whereas the magnetic increment is a curl.

For smooth divergence-free magnetic fields, summing the two vector fields in
\eqref{eq:continuous-operator-split} recovers the conservative MHD system
\eqref{eq:mhd-mass}--\eqref{eq:mhd-induction}. Moreover, each exact subflow
conserves the spatial integral of total energy under periodic boundary
conditions. The hydrodynamic contribution has flux
\((\Emech+p)\boldsymbol u\), while the magnetic contribution has Poynting
flux \(\Efield\times\boldsymbol B\); their sum is precisely the total-energy
flux in \eqref{eq:mhd-energy}.

The exact flows \(\varphi_H^\tau\) and \(\varphi_M^\tau\) generally do not
commute. Subsection~\ref{sec:algorithm} will combine numerical approximations of
these flows by symmetric Strang splitting~\cite{strang1968construction}.
Before doing so, Subsections
\ref{sec:spatial-prep}--\ref{sec:magnetic} construct the two spatially
discrete stage solvers and establish the properties inherited by each map.

\section{Numerical method}
\label{sec:numerical-method}

In this section, we construct a numerical discretization of the split system
\eqref{eq:continuous-operator-split} derived in
Section~\ref{sec:model-split}. We first introduce the tensor-product finite
element complex and its nodal realization. The hydrodynamic subflow is then
discretized by an entropy-stable discontinuous Galerkin spectral element
method (DGSEM) based on summation-by-parts flux differencing
\cite{gassner2016splitform,YangFu26}. The magnetic--velocity subflow is
discretized with compatible finite elements, drawing on the
\(H(\mathrm{curl})\)--\(H(\mathrm{div})\) framework for exactly
divergence-free, energy-stable MHD discretizations
\cite{hu2017stable,hu2019structure}. Finally, we couple the two update maps by
Strang splitting to obtain the fully discrete MHD algorithm.

\subsection{Discretization preliminaries}
\label{sec:spatial-prep}

\subsubsection{Elementwise de Rham complex}

Let \(K\) be a rectangular element obtained from the reference element
\(\widehat K=[-1,1]^2\) by the affine map
\(F_K:\widehat K\to K\). Let its side lengths in the \(x\)- and
\(y\)-directions be \(h_{x,K}\) and \(h_{y,K}\), respectively. The Jacobian
matrix of \(F_K\) and its determinant are
\[
  DF_K=\operatorname{diag}(h_{x,K}/2,h_{y,K}/2),
  \qquad J_K=\det(DF_K)=\frac{h_{x,K}h_{y,K}}{4}.
\]
For nonnegative integers \(a\) and \(b\), define
\[
  \mathbb Q^{a,b}(K)
  =\{q:q\circ F_K\in
  \mathbb P_a(\widehat x)\otimes\mathbb P_b(\widehat y)\},
  \qquad
  \mathbb Q^k(K)=\mathbb Q^{k,k}(K),
\]
where \(\mathbb P_k\) denotes the univariate polynomials of degree at most
\(k\).

For a general base degree \(k\ge0\), consider the four local spaces
\begin{align}
  \mathcal W^{k+1}(K)&=\mathbb Q^{k+1}(K)\subset H^1(K),
  \label{eq:local-h1-space}\\
  \mathcal N^k(K)&=
  \mathbb Q^{k,k+1}(K)\times\mathbb Q^{k+1,k}(K)
  \subset H(\mathrm{curl};K),
  \label{eq:local-hcurl-space}\\
  \mathcal R^k(K)&=
  \mathbb Q^{k+1,k}(K)\times\mathbb Q^{k,k+1}(K)
  \subset H(\mathrm{div};K),
  \label{eq:local-hdiv-space}\\
  \mathcal V^k(K)&=\mathbb Q^k(K)\subset L^2(K).
  \label{eq:local-l2-space}
\end{align}
We equip each space with nodal degrees of freedom. One such
choice is illustrated in Figure~\ref{fig:local-derham-dofs} for \(k=2\).
The scalar degrees of freedom are point evaluations, whereas the vector
degrees of freedom evaluate one coordinate component. On \(\partial K\), the
vector degrees of freedom determine the tangential trace of
\(\mathcal N^k(K)\) and the normal trace of \(\mathcal R^k(K)\);
the degrees of freedom of \(\mathcal V^k(K)\) remain element-local.

\begin{figure}[H]
  \centering
  \includegraphics[width=\textwidth]{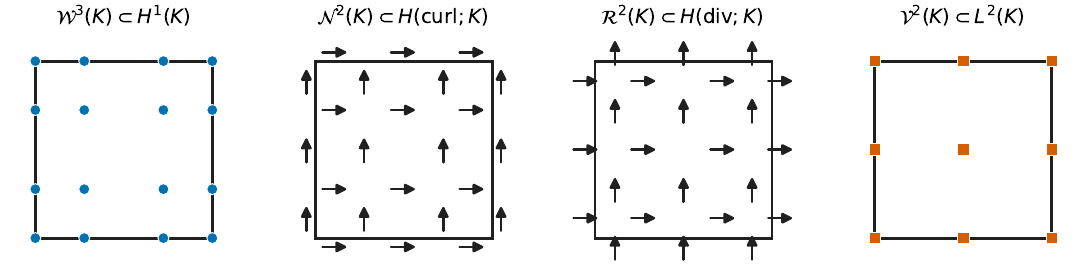}
  \caption{One choice of nodal degrees of freedom for the four local spaces
  with \(k=2\). Circles and squares denote scalar point evaluations in
  \(\mathcal W^3(K)\) and \(\mathcal V^2(K)\), respectively, whereas arrows
  denote coordinate-component evaluations in \(\mathcal N^2(K)\) and
  \(\mathcal R^2(K)\). Along \(\partial K\), the arrows represent tangential
  traces for \(\mathcal N^2(K)\) and normal traces for
  \(\mathcal R^2(K)\); for clarity, these trace arrows are drawn just outside
  the element and aligned with their boundary locations.}
  \label{fig:local-derham-dofs}
\end{figure}

The two vector spaces are the tensor-product N\'ed\'elec and
Raviart--Thomas spaces, respectively
\cite{nedelec1980,raviart1977,arnold2006feec}. On a single element they form
the rotation-equivalent de Rham sequences
\begin{equation}
\label{eq:local-de-rham-complexes}
\begin{array}{ccccc}
  \mathcal W^{k+1}(K) & \xrightarrow{\ \nabla\ } & \mathcal N^k(K)
  & \xrightarrow{\ \nabla\times\ } & \mathcal V^k(K) \\[1.5mm]
  \mathcal W^{k+1}(K) & \xrightarrow{\ \nabla^\perp\ } & \mathcal R^k(K)
  & \xrightarrow{\ \nabla\cdot\ } & \mathcal V^k(K),
\end{array}
\qquad
\nabla^\perp w=(-\partial_yw,\partial_xw)^T.
\end{equation}
Thus \(\nabla\times\mathcal N^k(K)\subset\mathcal V^k(K)\),
\(\nabla^\perp\mathcal W^{k+1}(K)\subset\mathcal R^k(K)\), and
\(\nabla\cdot\mathcal R^k(K)\subset\mathcal V^k(K)\). In particular, the
composition along either row vanishes.

\subsubsection{Assembly on the mesh}

Let \(\mathcal T_h\) be a conforming Cartesian mesh of the periodic domain
\(\Omega\). The global spaces are obtained by assembling the local degrees of
freedom with the continuity appropriate to each sequence in
\eqref{eq:local-de-rham-complexes}:
\begin{align}
  W_h^{k+1}
  &=\{w\in H^1(\Omega):w|_K\in\mathcal W^{k+1}(K)
  \quad\forall K\in\mathcal T_h\},
  \label{eq:h1-space}\\
  N_h^k
  &=\{\boldsymbol g\in H(\mathrm{curl};\Omega):
  \boldsymbol g|_K\in\mathcal N^k(K)
  \quad\forall K\in\mathcal T_h\},
  \label{eq:hcurl-space}\\
  R_h^k
  &=\{\boldsymbol r\in H(\mathrm{div};\Omega):
  \boldsymbol r|_K\in\mathcal R^k(K)
  \quad\forall K\in\mathcal T_h\},
  \label{eq:hdiv-space}\\
  V_h^k
  &=\{v\in L^2(\Omega):v|_K\in\mathcal V^k(K)
  \quad\forall K\in\mathcal T_h\}.
  \label{eq:l2-space}
\end{align}
Accordingly, scalar values are shared in \(W_h^{k+1}\), tangential traces are
shared in \(N_h^k\), and normal traces are shared in \(R_h^k\); no continuity
is imposed on \(V_h^k\). The elementwise differential maps assemble into
\[
  \nabla W_h^{k+1}\subset N_h^k,
  \qquad \nabla\times N_h^k\subset V_h^k,
  \qquad \nabla^\perp W_h^{k+1}\subset R_h^k,
  \qquad \nabla\cdot R_h^k\subset V_h^k.
\]

\subsubsection{Nodal bases, quadrature, and mass matrices}

We now choose bases compatible with the local degree-of-freedom layout. For
degree \(r\ge1\), let
\(\{\xi_i^{(r)},\omega_i^{(r)}\}_{i=0}^r\) be the one-dimensional
Gauss--Lobatto--Legendre (GLL) nodes and weights, and let
\(\{\ell_i^{(r)}\}_{i=0}^r\) be the associated Lagrange basis. For \(r=0\),
we use the constant basis \(\ell_0^{(0)}=1\) with the one-point rule
\(\xi_0^{(0)}=0\), \(\omega_0^{(0)}=2\). Let
\(\{\zeta_i^{(k)},\varpi_i^{(k)}\}_{i=0}^k\) be the \((k+1)\)-point
Gauss--Legendre rule, with degree-\(k\) Lagrange basis
\(\{g_i^{(k)}\}_{i=0}^k\).

On \(\widehat K\), the scalar spaces use the tensor bases
\[
  \mathcal W^{k+1}:\quad
  \ell_i^{(k+1)}(\widehat x)\ell_j^{(k+1)}(\widehat y),
  \qquad
  \mathcal V^k:\quad
  \ell_i^{(k)}(\widehat x)\ell_j^{(k)}(\widehat y).
\]
The N\'ed\'elec basis is
\begin{equation}
\label{eq:hcurl-basis}
\begin{aligned}
  g_i^{(k)}(\widehat x)\ell_j^{(k+1)}(\widehat y)\boldsymbol e_x,
  &\quad 0\le i\le k,\quad 0\le j\le k+1,\\
  \ell_i^{(k+1)}(\widehat x)g_j^{(k)}(\widehat y)\boldsymbol e_y,
  &\quad 0\le i\le k+1,\quad 0\le j\le k,
\end{aligned}
\end{equation}
and the Raviart--Thomas basis is its rotated counterpart,
\begin{equation}
\label{eq:hdiv-basis}
\begin{aligned}
  \ell_i^{(k+1)}(\widehat x)g_j^{(k)}(\widehat y)\boldsymbol e_x,
  &\quad 0\le i\le k+1,\quad 0\le j\le k,\\
  g_i^{(k)}(\widehat x)\ell_j^{(k+1)}(\widehat y)\boldsymbol e_y,
  &\quad 0\le i\le k,\quad 0\le j\le k+1.
\end{aligned}
\end{equation}
The physical scalar fields are obtained by the usual pullback, whereas the
vector fields use the covariant and contravariant Piola maps:
\[
  v_h\circ F_K=\widehat{v}_h,\qquad
  \boldsymbol a_h\circ F_K=DF_K^{-T}\widehat{\boldsymbol a}_h,
  \qquad
  \boldsymbol b_h\circ F_K=J_K^{-1}DF_K\widehat{\boldsymbol b}_h.
\]
These transformations correspond to the scalar, \(H(\mathrm{curl})\), and
\(H(\mathrm{div})\) spaces, respectively.

We pair each tensor-product basis with the tensor quadrature rule formed from
its nodal sets in the two coordinate directions. The quadrature points thus
coincide with the interpolation points, so the matrix representing the
corresponding discrete inner product in the nodal basis is diagonal. The
space-specific products used below are defined next. For the scalar \(L^2\)
space \(V_h^k\), the element product is
\begin{equation}
\label{eq:quad-scalar-element}
  (v_h,w_h)_{V,h,K}^{k}
  =|J_K|\sum_{i,j=0}^{k}\omega_i^{(k)}\omega_j^{(k)}
  \widehat v_h(\xi_i^{(k)},\xi_j^{(k)})
  \widehat w_h(\xi_i^{(k)},\xi_j^{(k)}),
\end{equation}
with mesh product
\begin{equation}
\label{eq:quad-scalar-mesh}
  (v_h,w_h)_{V,h}^{k}
  =\sum_{K\in\mathcal T_h}(v_h,w_h)_{V,h,K}^{k}.
\end{equation}
For the continuous scalar space \(W_h^{k+1}\), the degree-\(k+1\) GLL rule gives
\begin{equation}
\label{eq:quad-cg-element}
\begin{aligned}
  (v_h,w_h)_{W,h,K}^{k+1}
  ={}&|J_K|\sum_{i,j=0}^{k+1}
  \omega_i^{(k+1)}\omega_j^{(k+1)}\\
  &\quad\times
  \widehat v_h(\xi_i^{(k+1)},\xi_j^{(k+1)})
  \widehat w_h(\xi_i^{(k+1)},\xi_j^{(k+1)}).
\end{aligned}
\end{equation}
Since assembly only identifies coincident GLL nodes, the global mass matrix
remains diagonal. We denote the assembled product by
\begin{equation}
\label{eq:quad-cg-mesh}
  (v_h,w_h)_{W,h}^{k+1}
  =\sum_{K\in\mathcal T_h}(v_h,w_h)_{W,h,K}^{k+1}.
\end{equation}

For \(N_h^k\), the matching tensor rule is Gauss--GLL for the
\(x\)-component and GLL--Gauss for the \(y\)-component. If
\(\widehat{\boldsymbol a}_h=(\widehat a_x,\widehat a_y)^T\) and
\(\widehat{\boldsymbol b}_h=(\widehat b_x,\widehat b_y)^T\), then
\begin{equation}
\label{eq:quad-hcurl-element}
\begin{aligned}
  (\boldsymbol a_h,\boldsymbol b_h)_{N,h,K}^{k}
  ={}&\frac{h_{y,K}}{h_{x,K}}
  \sum_{i=0}^{k}\sum_{j=0}^{k+1}
  \varpi_i^{(k)}\omega_j^{(k+1)}
  \widehat a_x(\zeta_i^{(k)},\xi_j^{(k+1)})
  \widehat b_x(\zeta_i^{(k)},\xi_j^{(k+1)})\\
  &+\frac{h_{x,K}}{h_{y,K}}
  \sum_{i=0}^{k+1}\sum_{j=0}^{k}
  \omega_i^{(k+1)}\varpi_j^{(k)}
  \widehat a_y(\xi_i^{(k+1)},\zeta_j^{(k)})
  \widehat b_y(\xi_i^{(k+1)},\zeta_j^{(k)}).
\end{aligned}
\end{equation}
The two blocks are diagonal because their quadrature points coincide with the
two component bases in \eqref{eq:hcurl-basis}. The assembled product is
\begin{equation}
\label{eq:quad-hcurl-mesh}
  (\boldsymbol a_h,\boldsymbol b_h)_{N,h}^{k}
  =\sum_{K\in\mathcal T_h}
  (\boldsymbol a_h,\boldsymbol b_h)_{N,h,K}^{k}.
\end{equation}
This is the Gauss-point mass-lumping construction for edge elements
\cite{cohen1998mass}.

The subproblem discretizations below select the degree appropriate to each
field and suppress the degree superscript on these products. In every case in
which an inverse mass is required, the basis and quadrature are paired as
above, so the inverse action reduces to pointwise division.

\subsection{Hydrodynamic subflow: entropy-stable DGSEM}
\label{sec:hydro}

The hydrodynamic substep advances the Euler subsystem
\eqref{eq:hydro-continuous-mass}--\eqref{eq:hydro-continuous-energy} using the
entropy-stable DGSEM and stagewise oscillation elimination (OE) described
below. The damping-based DG framework was introduced for scalar conservation
laws in \cite{lu2021oscillation} and extended to hyperbolic systems in
\cite{liu2022essentially}. The OEDG procedure was subsequently introduced in
\cite{peng2025oedg}; its entropy-stable DGSEM realization used here follows
\cite{YangFu26}. In this substep the magnetic field is frozen. The evolved
conservative state is
\[
  U=(\rho,m_x,m_y,m_z,\Emech)^T,
  \qquad \boldsymbol m=\rho\boldsymbol u,
\]
with
\[
  p=(\gamma-1)\left(\Emech
  -\frac{m_x^2+m_y^2+m_z^2}{2\rho}\right).
\]
Thus the spatial dependence is two-dimensional, while all three velocity
components enter the pressure and energy. The physical Euler fluxes used in the
hydrodynamic step are
\[
F_x(U)=
\begin{pmatrix}
m_x\\
m_xu_x+p\\
m_yu_x\\
m_zu_x\\
(\Emech+p)u_x
\end{pmatrix},
\qquad
F_y(U)=
\begin{pmatrix}
m_y\\
m_xu_y\\
m_yu_y+p\\
m_zu_y\\
(\Emech+p)u_y
\end{pmatrix}.
\]
The \(z\)-momentum component is advected by the in-plane velocity; the pressure
term appears only in the momentum component aligned with the spatial direction
\(d\).

For the hydrodynamic subproblem, choose a polynomial degree \(p\ge1\) and
take \(k=p\) in the discontinuous space \(V_h^k\) introduced in
Section~\ref{sec:spatial-prep}. The conservative variables are approximated by
\[
  U_h\in [V_h^p]^5 .
\]
In this subsection, the parenthesized degree label on the GLL nodes, weights,
and basis functions is suppressed for readability.
Specifically, on each element \(K\), the DGSEM unknown is stored at the
tensor-product GLL nodes:
\[
  U_h|_K(x,y)=
  \sum_{i,j=0}^p U^K_{ij}\,
  \ell_i(\widehat x)\ell_j(\widehat y),
  \qquad U^K_{ij}\in\mathbb R^5 .
\]
The scalar product \eqref{eq:quad-scalar-element} gives a diagonal lumped mass
matrix on each element. We denote the diagonal entry associated with the nodal
degree of freedom \((i,j)\) by
\[
  M^K_{ij}:=|J_K|\omega_i\omega_j .
\]
For the five-component Euler state, this entry is applied componentwise.
For the one-dimensional GLL basis define
\[
  D_{ij}=\ell_j'(\xi_i),\qquad
  M=\operatorname{diag}(\omega_0,\ldots,\omega_p),\qquad
  Q=MD .
\]
Here \(D\), \(M\), and \(Q\) are \((p+1)\times(p+1)\) matrices.
The GLL derivative satisfies the summation-by-parts identity
\begin{equation}
\label{eq:euler-sbp}
  Q+Q^T=B,\qquad
  B=\operatorname{diag}(-1,0,\ldots,0,1),
\end{equation}
which is the discrete integration-by-parts mechanism behind conservation and
entropy stability.

\subsubsection{Entropy-conservative flux differencing and interface dissipation}

For the volume discretization, we use Chandrashekar's logarithmic-average
entropy-conservative flux \cite{chandrashekar2013kinetic}, written here in its
2.5D form. For a scalar \(a\) with two admissible states \(a_L\) and \(a_R\),
define the arithmetic and logarithmic means
\[
  \{\!\{a\}\!\}=\frac{a_L+a_R}{2},
  \qquad
  a^{\log}=\frac{a_R-a_L}{\log a_R-\log a_L},
\]
with the continuous extension \(a^{\log}=a_L\) when \(a_L=a_R\). Let
\(\beta=\rho/(2p)\),
\[
  \widehat p=\frac{\{\!\{\rho\}\!\}}{2\{\!\{\beta\}\!\}},
  \qquad
  \widehat h=
  \frac{1}{2(\gamma-1)\beta^{\log}}
  -\frac12\{\!\{|\boldsymbol u|^2\}\!\}
  +|\{\!\{\boldsymbol u\}\!\}|^2
  +\frac{\widehat p}{\rho^{\log}} .
\]
For \(d\in\{x,y\}\), the two-point flux is
\begin{equation}
\label{eq:euler-chandrashekar-flux}
  F_d^{\mathrm{ec}}(U_L,U_R)=
  \begin{pmatrix}
  \rho^{\log}\{\!\{u_d\}\!\}\\
  \rho^{\log}\{\!\{u_x\}\!\}\{\!\{u_d\}\!\}+\delta_{xd}\widehat p\\
  \rho^{\log}\{\!\{u_y\}\!\}\{\!\{u_d\}\!\}+\delta_{yd}\widehat p\\
  \rho^{\log}\{\!\{u_z\}\!\}\{\!\{u_d\}\!\}\\
  \rho^{\log}\{\!\{u_d\}\!\}\widehat h
  \end{pmatrix}.
\end{equation}
This flux is symmetric, consistent, and entropy conservative. With the
Euler entropy variables \(V=\partial\eta/\partial U\) and entropy potentials
\(\psi_d=V^TF_d-q_d\), it satisfies Tadmor's identity
\[
  (V_R-V_L)^T F_d^{\mathrm{ec}}(U_L,U_R)
  =\psi_d(U_R)-\psi_d(U_L).
\]

For a Cartesian rectangle it is convenient to absorb the constant metric factors
into contravariant fluxes and define
\[
  \mathcal F_x(U_L,U_R)=\frac{h_{y,K}}{2}F_x^{\mathrm{ec}}(U_L,U_R),
  \qquad
  \mathcal F_y(U_L,U_R)=\frac{h_{x,K}}{2}F_y^{\mathrm{ec}}(U_L,U_R).
\]
The flux-differencing volume residual is
\begin{equation}
\label{eq:euler-volume-residual}
\begin{aligned}
  R^{K,\mathrm{vol}}_{ij}
  ={}&2\sum_{r=0}^p Q_{ir}\omega_j\,
      \mathcal F_x(U^K_{ij},U^K_{rj})  +2\sum_{s=0}^p \omega_i Q_{js}\,
      \mathcal F_y(U^K_{ij},U^K_{is}) .
\end{aligned}
\end{equation}

We take the local Lax--Friedrichs flux as the numerical interface flux. For
\(d\in\{x,y\}\),
\begin{equation}
\label{eq:euler-lf-flux}
  \widehat F^{\mathrm{LF}}_d(U^-,U^+)
  =\frac12\left(F_d(U^-)+F_d(U^+)\right)
  -\frac12\lambda_d(U^-,U^+)(U^+-U^-),
\end{equation}
where
\[
  \lambda_d(U^-,U^+)
  =\max_{\sigma\in\{-,+\}}\left(
  |u_d^\sigma|+c^\sigma\right),
  \qquad c^\sigma=\sqrt{\gamma p^\sigma/\rho^\sigma}.
\]
The traces \(U^-\) and \(U^+\) are ordered in the positive coordinate direction
on each Cartesian face; on periodic boundaries, the exterior trace is taken
from the periodic neighbor.

Let
\[
  \widetilde F_x=\frac{h_{y,K}}{2}F_x(U_h),\qquad
  \widetilde F_y=\frac{h_{x,K}}{2}F_y(U_h),
\]
and define the corresponding contravariant numerical fluxes by
\[
  \widetilde F_x^*=\frac{h_{y,K}}{2}\widehat F_x^{\mathrm{LF}},
  \qquad
  \widetilde F_y^*=\frac{h_{x,K}}{2}\widehat F_y^{\mathrm{LF}} .
\]
Set \(\Delta\widetilde F_x=\widetilde F_x-\widetilde F_x^*\) and
\(\Delta\widetilde F_y=\widetilde F_y-\widetilde F_y^*\). The nodal DGSEM
update can be expressed as: for \(0\le i,j\le p\),
\begin{equation}
\label{eq:euler-nodal-index}
\begin{aligned}
  M^K_{ij}\frac{dU^K_{ij}}{dt}+R^{K,\mathrm{vol}}_{ij}
  &+\omega_i\!\left[(\Delta\widetilde F_y)_{i,0}\delta_{j,0}
  -(\Delta\widetilde F_y)_{i,p}\delta_{j,p}\right]\\
  &+\omega_j\!\left[(\Delta\widetilde F_x)_{0,j}\delta_{i,0}
  -(\Delta\widetilde F_x)_{p,j}\delta_{i,p}\right]=0,
\end{aligned}
\end{equation}
where $\delta_{i,j}$ is the Kronecker delta.
We also write
\eqref{eq:euler-nodal-index} abstractly as
\begin{equation}
\label{eq:euler-semidiscrete}
  M^K_{ij}\frac{dU^K_{ij}}{dt}
  +R^K_{ij}(U_h)=0,\qquad 0\le i,j\le p.
\end{equation}
Here \(R^K_{ij}\) combines the volume and boundary contributions.

\begin{theorem}[Euler DGSEM conservation and entropy stability]
\label{thm:euler-dgsem-entropy}
On a conforming Cartesian rectangular mesh with periodic boundary conditions,
assume that the nodal states satisfy \(\rho_h>0\) and \(p_h>0\). Then the update
\eqref{eq:euler-nodal-index} conserves the global Euler variables:
\[
  \frac{d}{dt}\sum_{K\in\mathcal T_h}\sum_{i,j=0}^p
  M^K_{ij}U^K_{ij}=0.
\]
If, in addition, the interface flux satisfies Tadmor's inequality
\begin{equation}
\label{eq:euler-interface-entropy-condition}
  (V^+-V^-)^T\widehat F_n(U^-,U^+)
  -\bigl(\psi_n(U^+)-\psi_n(U^-)\bigr)\le 0,
\end{equation}
then the update satisfies the discrete entropy inequality
\[
  \frac{d}{dt}\sum_{K\in\mathcal T_h}\sum_{i,j=0}^p
  M^K_{ij}\eta(U^K_{ij})\le 0 .
\]
\end{theorem}
The proof is the standard SBP flux-differencing argument: symmetry and
Tadmor's identity handle the volume terms, the interface entropy production is
nonpositive by \eqref{eq:euler-interface-entropy-condition}, and periodic
Cartesian faces cancel in conservative pairs. For the classical
Lax--Friedrichs flux \eqref{eq:euler-lf-flux}, a sufficient condition for the
interface entropy inequality \eqref{eq:euler-interface-entropy-condition} is
that
\(\lambda_d\) bound the largest wave speed in the exact Riemann fan
\cite{guermond2016wavespeed}.

\subsubsection{Stage stabilization: oscillation elimination and positivity
preservation}
We discretize the semidiscrete Euler system \eqref{eq:euler-semidiscrete} in
time with an explicit Runge--Kutta method.
Each provisional Runge--Kutta stage is stabilized before it enters the next
residual evaluation. The stage treatment consists of an
oscillation-eliminating procedure followed by conservative
positivity-preserving scaling. We describe these two operations separately and
then combine them in the fully discrete update.

\paragraph{Oscillation-eliminating procedure.}
Following the OEDG construction \cite{peng2025oedg} and its entropy-stable
DGSEM realization \cite{YangFu26}, the OE procedure acts elementwise on a
hierarchical decomposition of the conservative state. For
\(0\le\ell\le p\), let
\(\Pi_K^\ell:\mathbb Q^p(K)\to\mathbb Q^\ell(K)\) denote the projection
defined by the degree-\(p\) GLL product:
\begin{equation}
\label{eq:euler-oe-projection}
  (\Pi_K^\ell w_h-w_h,v_h)_{V,h,K}^{p}=0
  \qquad\forall v_h\in\mathbb Q^\ell(K).
\end{equation}
We apply this projection componentwise to \(U_h\). With
\(\mathcal S_K^\ell=\Pi_K^\ell-\Pi_K^{\ell-1}\), the nested polynomial spaces
give the incremental decomposition
\begin{equation}
\label{eq:euler-oe-increments}
  U_h|_K
  =\Pi_K^0U_h+(\Pi_K^1-\Pi_K^0)U_h+\cdots
   +(\Pi_K^p-\Pi_K^{p-1})U_h
  =\Pi_K^0U_h+\sum_{\ell=1}^p\mathcal S_K^\ell U_h.
\end{equation}
The constant projection \(\Pi_K^0U_h\) is the quadrature-weighted element
average, while every increment \(\mathcal S_K^\ell U_h\) has zero element
average.
We use the direction-resolved Cartesian sensor with cross-line scaling
introduced in \cite{FuLiu26}. For \(I\in\{x,y\}\), let
\(F_{I,K}^{-}\) and \(F_{I,K}^{+}\) be the two faces of \(K\) normal to
coordinate \(I\), and let \(h_{I,K}\) be the corresponding element width.
For a scalar field \(w_h\), define the normal-jump measure
\begin{equation}
\label{eq:euler-oe-jump}
  J_{r,I,K}(w_h)
  =\sum_{F\in\{F_{I,K}^{-},F_{I,K}^{+}\}}
  \left(
  \frac{1}{|F|}\int_F
  \left[\!\left[\partial_I^r w_h\right]\!\right]^2\,\dd s
  \right)^{1/2},
  \qquad 0\le r\le p .
\end{equation}
For each direction \(I\), let \(\mathcal X_{I,K}\) denote the row or column of
elements through \(K\) that extends in direction \(I\), with the transverse
coordinate fixed. Let \(\overline w_\Omega\) be the global mean
and \(\Delta_\Omega(w_h)=\|w_h-\overline w_\Omega\|_{L^\infty(\Omega)}\).
\begin{samepage}
The cross-line mean and amplitude are
\begin{equation}
\label{eq:euler-oe-amplitude}
\begin{aligned}
  \overline w_{I,K}
  &=\frac{\displaystyle\sum_{K'\in\mathcal X_{I,K}}
       \int_{K'}w_h\,\dd x}
      {\displaystyle\sum_{K'\in\mathcal X_{I,K}}|K'|},\\
  \Delta_{I,K}(w_h)
  &=\max\!\left\{
    \max_{K'\in\mathcal X_{I,K}}
      \|w_h-\overline w_{I,K}\|_{L^\infty(K')},
    10^{-6}\Delta_\Omega(w_h)\right\}.
\end{aligned}
\end{equation}
\end{samepage}
The global term supplies a small floor for a nearly constant coordinate line.
Figure~\ref{fig:euler-oe-cross-line-scaling} contrasts this construction with
a single domain-wide normalization. For the Euler state, the cross-line
amplitude is shared by all conservative components:
\begin{equation}
\label{eq:euler-oe-shared-amplitude}
  \Delta_{I,K}(U_h)
  =\max_{1\le q\le5}\Delta_{I,K}(U_h^{(q)}).
\end{equation}

\begin{figure}[H]
  \centering
  \resizebox{\linewidth}{!}{%
    \begin{tikzpicture}[
  x=0.86cm,
  y=0.86cm,
  every node/.style={font=\small},
  panel/.style={draw=black!70, line width=0.8pt},
  mesh/.style={draw=black!28, line width=0.35pt},
  guide/.style={line width=1.5pt},
  >=latex
]
  \definecolor{oeBlue}{RGB}{0,114,178}
  \definecolor{oeOrange}{RGB}{230,159,0}
  \definecolor{oeVermillion}{RGB}{213,94,0}

  \begin{scope}
    \fill[black!3] (0,0) rectangle (7,5);
    \fill[oeBlue!18] (1,2) rectangle (2,3);
    \fill[oeVermillion!70] (5,4) rectangle (6,5);
    \draw[mesh,step=1] (0,0) grid (7,5);
    \draw[panel] (0,0) rectangle (7,5);

    \node[font=\normalsize] at (3.5,5.42) {Domain-wide scaling};
    \node[font=\normalsize] at (1.5,2.5) {$K$};
    \node[align=center, anchor=south] (excursion) at (5.5,3.25)
      {largest excursion};
    \draw[->, oeVermillion, line width=0.9pt]
      (excursion.north) -- (5.5,3.96);
    \draw[->, dashed, black!65, line width=0.8pt]
      (5.15,3.92) to[bend right=18]
      node[pos=0.52, above, fill=white, inner sep=1pt]
      {$\Delta_\Omega(w_h)$}
      (1.92,3.02);
    \node[anchor=north, align=center] at (3.5,-0.28)
      {one scale shared by all elements};
  \end{scope}

  \begin{scope}[xshift=9.1cm]
    \fill[black!3] (0,0) rectangle (7,5);
    \fill[oeBlue!18] (0,2) rectangle (7,3);
    \fill[oeOrange!24] (2,0) rectangle (3,5);
    \fill[black!72] (2,2) rectangle (3,3);
    \draw[mesh,step=1] (0,0) grid (7,5);
    \draw[panel] (0,0) rectangle (7,5);

    \draw[guide,oeBlue] (0,2.5) -- (7,2.5);
    \draw[guide,oeOrange,dashed] (2.5,0) -- (2.5,5);
    \node[font=\normalsize] at (3.5,5.42) {Cross-line scaling};
    \node[text=white,font=\normalsize] at (2.5,2.5) {$K$};

    \node[anchor=south, fill=white, fill opacity=0.82, text opacity=1,
          inner sep=1.5pt, text=oeBlue!80!black] at (4.75,2.56)
      {$\mathcal X_{x,K}$};
    \node[rotate=90, anchor=center, fill=white, fill opacity=0.82,
          text opacity=1, inner sep=1.5pt, text=oeOrange!70!black]
      at (2.22,3.75) {$\mathcal X_{y,K}$};
    \node[anchor=north, align=center] at (3.5,-0.28)
      {direction-matched scales through $K$};
  \end{scope}
\end{tikzpicture}%
  }
  \caption{Domain-wide and direction-matched normalization for the OE
  indicator. Left: one amplitude \(\Delta_\Omega(w_h)\) is shared by every
  element, so an excursion far from \(K\) may control the scale. Right: the
  cross-line amplitudes for \(K\) are computed separately along the horizontal
  line \(\mathcal X_{x,K}\) and vertical line \(\mathcal X_{y,K}\). Accordingly,
  jumps across faces normal to \(x\) and \(y\) are normalized by variation along
  the matching coordinate direction. The schematic shows a scalar field;
  for the Euler state, the largest amplitude among the five conservative
  components supplies their common denominator.}
  \label{fig:euler-oe-cross-line-scaling}
\end{figure}
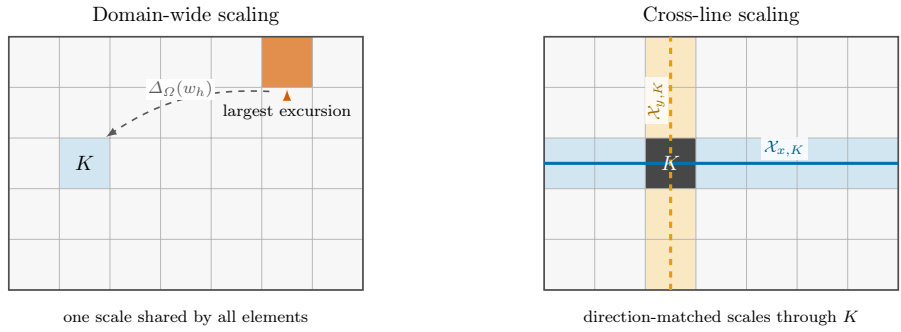

Writing \(U_h=(U_h^{(1)},\ldots,U_h^{(5)})^T\), we normalize the componentwise
jump measures by the shared amplitude \eqref{eq:euler-oe-shared-amplitude} and
then take their maximum:
\begin{equation}
\label{eq:euler-oe-sigma}
  \sigma_{r,I,K}(U_h)
  =\max_{1\le q\le5}
   \frac{(2r+1)h_{I,K}^{r}}{2(r!)}
   \frac{J_{r,I,K}(U_h^{(q)})}{\Delta_{I,K}(U_h)},
  \qquad 0\le r\le p,
\end{equation}
where a component that is constant to roundoff contributes zero. Sharing the
denominator prevents a nearly constant component from being interpreted as an
order-one feature merely because its own amplitude is small. An \(x\)-normal
jump is scaled by the line varying in \(x\) at fixed \(y\), and analogously in
the \(y\)-direction.

Let \(\overline U_K\) be the quadrature-weighted cell average and let \(c_K\)
be the sound speed recovered from this state. The direction-matched Euler rate
is
\begin{equation}
\label{eq:euler-oe-rate}
  \lambda_{E,I,K}
  =(2p+1)\frac{|\overline u_{I,K}|+c_K}{h_{I,K}}.
\end{equation}
For a stage time step \(\tau\), define the damping factor
\begin{equation}
\label{eq:euler-oe-factor}
\begin{aligned}
  \mathcal D_{\ell,K}
  &=\exp\!\left[-s_H\tau\Theta_{\ell,K}\right],\quad\text{where }
  \Theta_{\ell,K}
  =\sum_{I\in\{x,y\}}\lambda_{E,I,K}
    \sum_{r=0}^{\ell}\sigma_{r,I,K}.
\end{aligned}
\end{equation}
The Cartesian OE update damps each hierarchical increment independently:
\begin{equation}
\label{eq:euler-oe-map}
  (\mathcal P_{\mathrm{OE}}U_h)|_K
  =\Pi_K^0U_h+\sum_{\ell=1}^p
   \mathcal D_{\ell,K}\mathcal S_K^\ell U_h.
\end{equation}
where \(s_H\ge0\) is the hydrodynamic-stage OE scale factor. For all nonsmooth
experiments reported below, we take \(p=2\) and set
\(s_H=0.02\). Since the
constant projection is unchanged and every damped increment has zero average,
\(\mathcal P_{\mathrm{OE}}\) preserves the conservative element average.

\paragraph{Positivity-preserving limiter.}
After OE, we apply the conservative scaling limiter of Zhang and Shu
\cite{zhang2010positivity}. Let \(U_h^\star\) denote the OE-filtered stage
candidate and define its quadrature-weighted average on \(K\) by
\begin{equation}
\label{eq:euler-pp-average}
  \overline U_K
  =\frac{1}{|K|}\sum_{i,j=0}^p M^K_{ij}U_{ij}^\star,
  \qquad
  \overline\rho_K=(\overline U_K)_\rho,
  \qquad
  \overline p_K=p(\overline U_K).
\end{equation}
The implementation uses the following roundoff-aware floors:
\begin{equation}
\label{eq:euler-pp-floors}
  \varepsilon_\rho=64\epsilon_{\mathrm{mach}},
  \qquad
  \varepsilon_{p,K}
  =\max\!\left\{64\epsilon_{\mathrm{mach}},
  8192\epsilon_{\mathrm{mach}}S_{p,K}\right\},
\end{equation}
where \(\epsilon_{\mathrm{mach}}\) is machine precision and \(S_{p,K}\) is
the maximum, over the element average and the density-limited GLL states, of
the pressure-cancellation scale
\[
  S_p(U)=\lvert E_{\mathrm{mech}}\rvert
  +\frac{\lvert\boldsymbol m\rvert^2}
         {2\max\{\rho,\varepsilon_\rho\}}.
\]
Provided that the average state is admissible, density is first limited by
\begin{equation}
\label{eq:euler-pp-density}
\begin{aligned}
  \rho_{\min,K}
  &=\min_{0\le i,j\le p}\rho_{ij}^\star,\\
  \theta_{\rho,K}
  &=\min\!\left\{1,
    \frac{\overline\rho_K-\varepsilon_\rho}
         {\overline\rho_K-\rho_{\min,K}}\right\},\\
  \rho_{ij}^{(\rho)}
  &=\overline\rho_K
    +\theta_{\rho,K}(\rho_{ij}^\star-\overline\rho_K).
\end{aligned}
\end{equation}
Here and in the pressure step below, \(\theta=1\) when the corresponding nodal
minimum already satisfies its floor; otherwise, the displayed ratio is
clipped to \([0,1]\).
The intermediate state
\(U_{ij}^{(\rho)}\) replaces only the density of \(U_{ij}^\star\) by
\(\rho_{ij}^{(\rho)}\); its three momentum components and mechanical energy
are unchanged. The pressure is then recomputed from this intermediate state,
and the full conservative vector is scaled according to
\begin{equation}
\label{eq:euler-pp-pressure}
\begin{aligned}
  p_{\min,K}
  &=\min_{0\le i,j\le p}p\!\left(U_{ij}^{(\rho)}\right),\\
  \theta_{p,K}
  &=\min\!\left\{1,
    \frac{\overline p_K-\varepsilon_{p,K}}
         {\overline p_K-p_{\min,K}}\right\},\\
  (\mathcal P_{\mathrm{PP}}U_h^\star)_{ij}
  &=\overline U_K
    +\theta_{p,K}\left(U_{ij}^{(\rho)}-\overline U_K\right).
\end{aligned}
\end{equation}
The density scaling preserves the density average, and the second scaling
preserves the complete conservative average. For an average state above the
prescribed floors, \(\mathcal P_{\mathrm{PP}}\) therefore enforces the nodal
conditions \(\rho_h\ge\varepsilon_\rho\) and
\(p_h\ge\varepsilon_{p,K}\) without changing the element average. For a
vacuum-like average within the roundoff tolerance, the implementation instead
collapses the corresponding element to its average. A genuinely inadmissible
average cannot be repaired by conservative scaling and therefore causes the
Euler substep to be rejected.

The complete stage-stabilization map is
\begin{equation}
\label{eq:euler-stage-map}
  \boxed{
  \mathcal P
  =\mathcal P_{\mathrm{PP}}\circ\mathcal P_{\mathrm{OE}}:
  \quad
  \text{OE damping}
  \;\longrightarrow\;
  \text{positivity-preserving scaling}}
\end{equation}
and preserves the conservative element average.

\subsubsection{SSP-RK(3,3) update}

Let \(\mathcal L_E(U_h)\) denote the mass-inverted semidiscrete Euler operator
defined by \eqref{eq:euler-semidiscrete}, so that
\[
  \frac{dU_h}{dt}=\mathcal L_E(U_h).
\]
For a time step \(\tau>0\), the implementation combines the stage map
\eqref{eq:euler-stage-map} with the third-order strong-stability-preserving
Runge--Kutta method (SSP-RK(3,3)) \cite{gottlieb2001strong}:
\begin{subequations}
\label{eq:euler-ssprk33}
\begin{align}
  U^{(1)}
  &=\mathcal P\!\left(U^n+\tau\mathcal L_E(U^n)\right),\\
  U^{(2)}
  &=\mathcal P\!\left(\frac34 U^n
    +\frac14\left[U^{(1)}
    +\tau\mathcal L_E(U^{(1)})\right]\right),\\
  U^{n+1}
  &=\mathcal P\!\left(\frac13 U^n
    +\frac23\left[U^{(2)}
    +\tau\mathcal L_E(U^{(2)})\right]\right).
\end{align}
\end{subequations}

\begin{theorem}[Cell-average admissibility of the Euler update]
\label{thm:euler-cell-average-admissibility}
Let
\[
  \mathcal G_E
  =\left\{U:\rho>0,\quad
  E_{\mathrm{mech}}-\frac{|\boldsymbol m|^2}{2\rho}>0\right\}
\]
be the Euler admissible set. Suppose that the GLL nodal states entering each
forward-Euler evaluation belong to \(\mathcal G_E\), and that the LLF
dissipation on every face satisfies
\[
  \lambda_d(U^-,U^+)
  \ge \max_{\sigma\in\{-,+\}}
  \bigl(|u_d^\sigma|+c^\sigma\bigr).
\]
Then there exists \(\tau_0>0\) such that, for every \(0<\tau\le\tau_0\), each raw
forward-Euler update has an admissible element average. In particular, the
standard rectangular-mesh LLF positivity CFL condition
\cite{zhang2010positivity} is sufficient.
\end{theorem}

Indeed, summing the flux-differencing update over the GLL nodes and using the
SBP identity reduces the element-average equation to the usual conservative
face-flux balance. The positive tensor-product GLL weights express this update
as a convex combination of admissible nodal states and Euler
Lax--Friedrichs split states under the stated CFL condition. Convexity of
\(\mathcal G_E\) then gives admissibility of the forward-Euler average and of
the SSP-RK convex combinations. Finally, both OE and positivity scaling
preserve the element average, while the latter restores nodal admissibility.
The LLF requirement in this argument is the one-state Euler
Lax--Friedrichs splitting bound; it is distinct from the exact-Riemann-fan
condition used in Theorem~\ref{thm:euler-dgsem-entropy}.

Before applying \(\mathcal P\), the conservative average of every raw stage
candidate is checked for admissibility. If any element fails this check, the
entire Euler substep is rejected, the solution is restored to \(U^n\), and the
substep is repeated with \(\tau\) reduced by a factor of two. For a strictly
admissible input, Theorem~\ref{thm:euler-cell-average-admissibility} guarantees
a positive timestep threshold, so repeated halving eventually reaches the
admissible regime.

Algorithm~\ref{alg:hydro-step} summarizes the hydrodynamic Euler substep
defined in this section.

\begin{algorithm}[H]
\caption{Hydrodynamic Euler substep \(\Phi_H^\tau\)}
\label{alg:hydro-step}
\begin{algorithmic}[1]
\Require \((\rho_h,\boldsymbol m_h,E_{\mathrm{mech},h},\boldsymbol B_h)^n\)
         and time step \(\tau\).
\State Freeze \(\boldsymbol B_h\) and set
       \(U_h^n=(\rho_h,\boldsymbol m_h,E_{\mathrm{mech},h})^T\in [V_h^p]^5\).
\State Form the DGSEM residual \(\mathcal L_E(U_h)\) from
       \eqref{eq:euler-semidiscrete}, using
       \eqref{eq:euler-chandrashekar-flux} in the volume and
       \eqref{eq:euler-lf-flux} on element interfaces.
\State Advance \(U_h\) with the stabilized SSP-RK scheme
       \eqref{eq:euler-ssprk33}, applying
       \(\mathcal P=\mathcal P_{\mathrm{PP}}\circ
       \mathcal P_{\mathrm{OE}}\) to every raw stage candidate.
\State Set
       \(U_h^{\mathrm{out}}
       =(\rho_h,\boldsymbol m_h,E_{\mathrm{mech},h})^T\in [V_h^p]^5\).
\Ensure \((U_h^{\mathrm{out}},\boldsymbol B_h^n)\), with
        \(\boldsymbol B_h\) unchanged.
\end{algorithmic}
\end{algorithm}

\subsection{Magnetic--velocity subflow: compatible finite elements}
\label{sec:magnetic}

We now discretize the magnetic--velocity subsystem
\eqref{eq:mag-rhoe}--\eqref{eq:mag-constitutive}. Throughout this substep,
the discrete density \(\rho_h\) is held fixed. The ideal relation
\(\rho\partial_t e=0\) becomes \(\rho_h\partial_t e_h=0\); when artificial
resistivity is enabled, it is extended by the Ohmic-heating source
\(\rho_h\partial_t e_h=\eta_h|\Jh|^2\) in the compatible weak form given
below. The separate velocity-OE
correction also returns its kinetic-energy loss to \(e_h\), as described in
Section~\ref{sec:magnetic-stage-stabilization}.

To retain the curl structure of the induction equation, we choose a compatible
degree \(m\in\{p-1,p\}\), take \(k=m\) in the de Rham families of
Section~\ref{sec:spatial-prep}, and introduce the velocity and magnetic spaces
\begin{equation}
\label{eq:magnetic-evolved-spaces}
  X_h^p=(V_h^p)^3,
  \qquad
  \mathcal B_h^m=R_h^m\times V_h^m.
\end{equation}
Thus \(\uh(t)\in X_h^p\), while
\(\Bh(t)=(\boldsymbol B_{\parallel,h},B_{z,h})\in\mathcal B_h^m\), and the
frozen density \(\rho_h\) and specific internal energy \(e_h(t)\) both belong
to \(V_h^p\).

The electric field and current density are reconstructed in the auxiliary
space
\begin{equation}
\label{eq:magnetic-auxiliary-space}
  \mathcal A_h^m=N_h^m\times W_h^{m+1}.
\end{equation}
Accordingly, \(\Eh,\Jh\in\mathcal A_h^m\), with their in-plane components
in the tensor-product N\'ed\'elec space and their out-of-plane components in
the continuous scalar space. The 2.5D curl maps these two product spaces in
the natural direction
\[
  \nabla\times\mathcal A_h^m\subset\mathcal B_h^m.
\]

For \(\boldsymbol v_h,\boldsymbol z_h\in X_h^p\), let
\((\cdot,\cdot)_{X,h}\) denote the componentwise degree-\(p\) DG--GLL
product. Its \(\rho_h\)-weighted form is
\begin{equation}
\label{eq:velocity-inner-product}
  (\rho_h\boldsymbol v_h,\boldsymbol z_h)_{X,h}
  =\sum_{d=1}^3(\rho_hv_{d,h},z_{d,h})_{V,h}.
\end{equation}
The corresponding weighted mass matrix is diagonal. For
\(\boldsymbol F_h,\boldsymbol G_h\in\mathcal A_h^m\), define the
mass-lumped auxiliary product
\begin{equation}
\label{eq:aux-inner-product}
  (\boldsymbol F_h,\boldsymbol G_h)_{\mathcal A,h}
  =
  (\boldsymbol F_{\parallel,h},\boldsymbol G_{\parallel,h})_{N,h}
  +(F_{z,h},G_{z,h})_{W,h}.
\end{equation}
Here \((\cdot,\cdot)_{N,h}\) is the mixed Gauss--GLL/GLL--Gauss rule
\eqref{eq:quad-hcurl-mesh}, and \((\cdot,\cdot)_{W,h}\) is the
degree-\(m+1\) CG--GLL rule \eqref{eq:quad-cg-mesh}. Both blocks are
diagonal on the Cartesian meshes considered here. Hence reconstructing
\(\Jh\) and \(\Eh\) requires only pointwise scaling.

For \(\boldsymbol C_h,\boldsymbol D_h\in\mathcal B_h^m\), we denote the
matched magnetic product by
\begin{equation}
\label{eq:magnetic-inner-product}
  (\boldsymbol C_h,\boldsymbol D_h)_{\mathcal B,h}
  =\sum_{K\in\mathcal T_h}
   Q_{\mathcal B,K}(\boldsymbol C_h\cdot\boldsymbol D_h),
\end{equation}
where \(Q_{\mathcal B,K}\) applies componentwise tensor-Gauss rules of
coordinate-wise exactness \(2m+3\) to the in-plane RT\(_m\) block and
\(2m+1\) to the out-of-plane \(\mathbb Q^m\) block. These rules exactly
integrate the corresponding mass products on affine rectangular elements. We
do not mass-lump the magnetic product. It defines both the weak curl and the
discrete magnetic energy, and its mass matrix is applied but never inverted.

The mass-lumped auxiliary quadrature defines the trilinear form
\begin{equation}
\label{eq:ideal-coupling-form}
  \mathcal C_h(\boldsymbol v_h,\boldsymbol B_h;\boldsymbol G_h)
  =(-\boldsymbol v_h\times\boldsymbol B_h,\boldsymbol G_h)_{\mathcal A,h}.
\end{equation}
The velocity and magnetic fields are evaluated at the auxiliary nodes. Thus,
evaluating \(\mathcal C_h\) is equivalent to interpolating
\(-\boldsymbol v_h\times\boldsymbol B_h\) nodally into
\(\mathcal A_h^m\) and pairing the result with \(\boldsymbol G_h\) in the
mass-lumped auxiliary product. The Ohm-law load uses \(\mathcal C_h\)
directly, whereas the Lorentz load is assembled as its algebraic transpose.
This pairing is the discrete mechanism behind the ideal energy exchange proved
below.

Finally, for a nonnegative elementwise constant resistivity
\(\eta_h\in V_h^0\), define the Ohmic-heating form
\(\mathcal H_{V,h}\) by
\begin{equation}
\label{eq:ohmic-heating-1}
  \mathcal H_{V,h}(\eta_h;\Jh,\Jh,v_h)
  =(\eta_h|\Jh|^2,v_h)_{V,h}
  \qquad\forall v_h\in V_h^p.
\end{equation}
Here all factors on the right-hand side are evaluated with the degree-\(p\)
DG--GLL rule. In particular,
\begin{equation}
\label{eq:ohmic-heating-compatibility}
  \mathcal H_{V,h}(\eta_h;\Jh,\Jh,1)
  =(\eta_h|\Jh|^2,1)_{V,h}.
\end{equation}
This identity pairs the resistive magnetic-energy loss with the internal-energy
gain.

\subsubsection{Semidiscrete differential-algebraic formulation}
The spatial discretization of the magnetic--velocity substep
\eqref{eq:mag-rhoe}--\eqref{eq:mag-constitutive} takes the following
semidiscrete differential-algebraic form. Given a density \(\rho_h\)
with positive nodal values and a nonnegative \(\eta_h\in V_h^0\), both held
fixed throughout the substep, find
\(\uh(t)\in X_h^p\), \(\Bh(t)\in\mathcal B_h^m\),
\(e_h(t)\in V_h^p\), and
\(\Jh(t),\Eh(t)\in\mathcal A_h^m\) such that
\begin{subequations}
\label{eq:magnetic-dae}
\begin{align}
  (\rho_h\partial_t\uh,\boldsymbol v_h)_{X,h}
  -\mathcal C_h(\boldsymbol v_h,\Bh;\Jh)&=0
  &&\forall \boldsymbol v_h\in X_h^p, \label{eq:semi-vel}\\
  (\partial_t\Bh,\boldsymbol C_h)_{\mathcal B,h}
  +(\nabla\times\Eh,\boldsymbol C_h)_{\mathcal B,h}&=0
  &&\forall \boldsymbol C_h\in\mathcal B_h^m, \label{eq:semi-B}\\
  (\rho_h\partial_t e_h,v_h)_{V,h}
  -\mathcal H_{V,h}(\eta_h;\Jh,\Jh,v_h)&=0
  &&\forall v_h\in V_h^p, \label{eq:ohmic-heating}\\
  (\Jh,\boldsymbol\chi_h)_{\mathcal A,h}
  -(\Bh,\nabla\times\boldsymbol\chi_h)_{\mathcal B,h}&=0
  &&\forall \boldsymbol\chi_h\in\mathcal A_h^m,
  \label{eq:mag-current-proj}\\
  (\Eh,\boldsymbol G_h)_{\mathcal A,h}
  -\mathcal C_h(\uh,\Bh;\boldsymbol G_h)
  -(\eta_h\Jh\cdot\boldsymbol G_h,1)_{V,h}&=0
  &&\forall \boldsymbol G_h\in\mathcal A_h^m.
  \label{eq:mag-electric-proj}
\end{align}
\end{subequations}
The first three equations evolve \(\uh\), \(\Bh\), and \(e_h\); the last two
determine \(\Jh\) and \(\Eh\) algebraically from the instantaneous
differential variables. The choice \(\eta_h=0\) recovers the ideal subflow,
for which \(\partial_t e_h=0\). Under periodic boundary conditions,
\eqref{eq:mag-current-proj} is the integration-by-parts
form of \(\Jh=\nabla\times\Bh\), while
\eqref{eq:mag-electric-proj} is the discretization of Ohm's law
\(\Efield=-\boldsymbol u\times\boldsymbol B+\eta_h\J\).

Because the internal-energy mass matrix and the Ohmic source use the same
DG--GLL rule, the third equation in \eqref{eq:magnetic-dae} reduces nodally to
\[
  \rho_h(x_q)\,\partial_t e_h(x_q)
  =\eta_h(x_q)|\Jh(x_q)|^2\ge0.
\]
Thus a positive nodal internal energy remains positive under the resistive
semidiscrete evolution. The same conclusion holds for the SSP-RK update,
whose forward-Euler source increments are nonnegative.

Because the Lorentz load is the algebraic transpose of
\eqref{eq:ideal-coupling-form}, the first equation represents
\(\rho_h\partial_t\uh+\Bh\times\Jh=0\) with frozen density.
The second equation is the compatible finite-element representation of
\(\partial_t\Bh+\nabla\times\Eh=0\). Componentwise,
\[
  -\nabla\times\Eh
  =
  \left(-\partial_yE_{z,h},\,
        \partial_xE_{z,h},\,
        -(\partial_xE_{y,h}-\partial_yE_{x,h})\right)^T .
\]
The in-plane magnetic increment is therefore generated entirely by the scalar
continuous field \(E_{z,h}\), while the \(B_{z,h}\) increment is generated by
the \(H(\mathrm{curl})\) in-plane electric field. Because these curls already
belong to \(R_h^m\) and \(V_h^m\), respectively, the implementation
applies the corresponding discrete derivative operators directly. The same
magnetic product appears in both terms of \eqref{eq:semi-B}; its mass matrix
therefore cancels algebraically and is not inverted.

\paragraph{Energy balance.}
The paired coupling and Ohmic forms yield an exact energy balance for the
semidiscrete scheme \eqref{eq:magnetic-dae}.
\begin{theorem}[Semidiscrete magnetic-step energy balance]
\label{thm:energy}
Under periodic boundary conditions and for fixed nodal density \(\rho_h>0\),
the semidiscrete system \eqref{eq:magnetic-dae} with fixed
\(\eta_h\in V_h^0\), \(\eta_h\ge0\), satisfies
\[
  \frac{d}{dt} E_{tot,h}(t)=0,
  \qquad
  E_{tot,h}(t)
  =
  (\rho_he_h,1)_{V,h}
  +\frac12\|\sqrt{\rho_h}\uh\|_{X,h}^2
  +\frac12\|\Bh\|_{\mathcal B,h}^2 .
\]
Moreover,
\begin{align*}
  \frac{d}{dt}\left[
  \frac12\|\sqrt{\rho_h}\uh\|_{X,h}^2
  +\frac12\|\Bh\|_{\mathcal B,h}^2\right]
  &=-(\eta_h|\Jh|^2,1)_{V,h}\le0,\\
  \frac{d}{dt}(\rho_he_h,1)_{V,h}
  &=(\eta_h|\Jh|^2,1)_{V,h}\ge0.
\end{align*}
\end{theorem}

\begin{proof}
Choose \(\boldsymbol v_h=\uh\) in \eqref{eq:semi-vel}. Since \(\rho_h\) is
fixed,
\[
  \frac{d}{dt}\frac12\|\sqrt{\rho_h}\uh\|_{X,h}^2
  =\mathcal C_h(\uh,\Bh;\Jh).
\]
Choose \(\boldsymbol C_h=\Bh\) in \eqref{eq:semi-B}. Using the current
projection \eqref{eq:mag-current-proj} with
\(\boldsymbol\chi_h=\Eh\),
\[
  \frac{d}{dt}\frac12\|\Bh\|_{\mathcal B,h}^2
  =-(\nabla\times\Eh,\Bh)_{\mathcal B,h}
  =-(\Jh,\Eh)_{\mathcal A,h}.
\]
Taking \(\boldsymbol G_h=\Jh\) in the resistive Ohm law
\eqref{eq:mag-electric-proj} gives
\[
  (\Jh,\Eh)_{\mathcal A,h}
  =\mathcal C_h(\uh,\Bh;\Jh)
   +(\eta_h|\Jh|^2,1)_{V,h}.
\]
The kinetic and magnetic contributions therefore satisfy the first balance.
Taking \(v_h=1\) in \eqref{eq:ohmic-heating} and using
\eqref{eq:ohmic-heating-compatibility} gives the second. Adding the two
identities yields
\[
  \frac{d}{dt}E_{tot,h}(t)=0,
\]
which proves the result.
\end{proof}

\paragraph{Divergence-free property.}
The compatible form of the in-plane magnetic increment gives
\[
  \partial_t\boldsymbol B_{\parallel,h}
  =
  \nabla^\perp E_{z,h}
  =
  (-\partial_yE_{z,h},\,\partial_xE_{z,h})^T\in R_h^m.
\]

\begin{theorem}[Preservation of the in-plane divergence-free constraint]
\label{thm:magnetic-divergence}
If the initial in-plane magnetic field \(\boldsymbol B_{\parallel,h}(0)\)
belongs to the global space \(R_h^m\) and is elementwise divergence-free,
then the semidiscrete magnetic update preserves this global
\(H(\mathrm{div})\) divergence-free property:
\[
  \boldsymbol B_{\parallel,h}(t)\in R_h^m,\qquad
  \nabla\cdot\boldsymbol B_{\parallel,h}(t)=0
  \quad\text{on each }K\in\mathcal T_h .
\]
The same statement holds for the explicit Runge--Kutta time stepping built from
the magnetic residual \eqref{eq:semi-B}.
\end{theorem}

Indeed, every in-plane increment is a rotated gradient of a continuous
\(E_{z,h}\). Its elementwise divergence vanishes by commutation of mixed
derivatives, and its normal trace is single valued across each mesh face.
Linear combinations of such increments give the Runge--Kutta statement.

\subsubsection{Stage stabilization: velocity OE and artificial resistivity}
\label{sec:magnetic-stage-stabilization}

As in the Euler substep, we introduce the stage stabilization before the
fully discrete update. Two complementary mechanisms are used. Velocity OE is
an element-local postprocessing map applied to every provisional Runge--Kutta
stage, whereas artificial resistivity enters the magnetic residual and its
paired Ohmic-heating equation. We describe these mechanisms in turn.

\paragraph{Velocity OE procedure.}
For each element \(K\), define the density-weighted collocated inner product
\begin{equation}
\label{eq:rho-weighted-inner-product}
  (\boldsymbol v_h,\boldsymbol z_h)_{\rho,K}
  =(\rho_h\boldsymbol v_h,\boldsymbol z_h)_{X,h,K}
  =\sum_{x_q\in K}M_q^K\rho_q
  \boldsymbol v_q\cdot\boldsymbol z_q .
\end{equation}
Let \(P_{\ell,K}^{\rho}\) be the orthogonal projection in this inner product
onto \([\mathbb Q^\ell(K)]^3\), set \(P_{-1,K}^{\rho}=0\), and introduce the
mutually
orthogonal velocity increments
\[
  \boldsymbol w_{\ell,K}
  =(P_{\ell,K}^{\rho}-P_{\ell-1,K}^{\rho})\uh,
  \qquad 0\le \ell\le p.
\]
Because the nodal density is frozen, these projectors remain fixed throughout
the magnetic substep.

The sensor uses the directional normal-jump construction and cross-line
scaling of the hydrodynamic OE operator. To prevent a small-amplitude velocity
component from dominating the sensor, define the shared amplitude
\[
  \Delta^u_{I,K}
  =\max_{d\in\{x,y,z\}}\Delta_{I,K}(u_{d,h}),
\]
where \(\Delta_{I,K}\) is given by \eqref{eq:euler-oe-amplitude}. With the
jump measure \eqref{eq:euler-oe-jump}, the velocity sensor is
\begin{equation}
\label{eq:magnetic-oe-sensor}
  \sigma^u_{r,I,K}
  =\max_{d\in\{x,y,z\}}
   \frac{(2r+1)h_{I,K}^{r}}{2(r!)}
   \frac{J_{r,I,K}(u_{d,h})}{\Delta^u_{I,K}},
  \qquad 0\le r\le p,
\end{equation}
where a component that is constant to roundoff contributes zero. The
direction-matched magnetic-stage rate is
\begin{equation}
\label{eq:magnetic-oe-rate}
  \lambda_{M,I,K}
  =\frac{2p+1}{h_{I,K}}
  \max_{x_q\in K}\left(
  |u_{I,h}(x_q)|+
  \frac{|\Bh(x_q)|}{\sqrt{\rho_h(x_q)}}
  \right).
\end{equation}
For a magnetic-stage time interval \(\tau\), define
\begin{equation}
\label{eq:magnetic-oe-factor}
\begin{aligned}
  \mathcal D_{\ell,K}
  &=\exp\!\left[-s_M\tau\Theta_{\ell,K}\right],\quad\text{where }
  \Theta_{\ell,K}
  =\sum_{I\in\{x,y\}}\lambda_{M,I,K}
    \sum_{r=0}^{\ell}\sigma^u_{r,I,K}.
\end{aligned}
\end{equation}
where \(s_M\ge0\) is the magnetic-stage OE scale factor. The postprocessing
map is
\begin{equation}
\label{eq:magnetic-oe-map}
\begin{aligned}
  \boldsymbol u_{h,q}^+
  &= (\boldsymbol w_{0,K})_q
     +\sum_{\ell=1}^p\mathcal D_{\ell,K}
       (\boldsymbol w_{\ell,K})_q,\\
  e_q^+
  &=e_q^-+\frac12\sum_{\ell=1}^p
    \left[1-\mathcal D_{\ell,K}^2\right]
    |(\boldsymbol w_{\ell,K})_q|^2 .
\end{aligned}
\end{equation}
This is also the exact solution over the artificial time interval
\(0\le s\le\tau\) of the local damping--heating system
\begin{equation}
\label{eq:magnetic-oe-continuous}
\begin{aligned}
  \rho_h\partial_s\boldsymbol u_h
  &=-\rho_h s_M\sum_{\ell=1}^p\Theta_{\ell,K}
    (P_{\ell,K}^{\rho}-P_{\ell-1,K}^{\rho})\boldsymbol u_h,\\
  \rho_h\partial_s e_h
  &=\rho_h s_M\sum_{\ell=1}^p\Theta_{\ell,K}
    \left|(P_{\ell,K}^{\rho}-P_{\ell-1,K}^{\rho})
    \boldsymbol u_h\right|^2,
\end{aligned}
\end{equation}
with the sensor coefficients and projectors frozen during the map. For all
nonsmooth experiments reported below, we take \(p=2\) and set \(s_M=0.02\).

\begin{proposition}[Magnetic velocity-OE map]
\label{prop:magnetic-oe-energy}
For fixed nodal density \(\rho_q>0\), the map
\eqref{eq:magnetic-oe-map} preserves element momentum, adds a nonnegative
increment to \(e_h\) at every node, and satisfies
\begin{equation}
\label{eq:magnetic-oe-energy-balance}
\begin{aligned}
  \sum_qM_q^K\rho_q\boldsymbol u_{h,q}^+
  &=\sum_qM_q^K\rho_q\boldsymbol u_{h,q}^-,\\
  \frac12\|\sqrt{\rho_h}\uh^-\|_{X,h,K}^2
  -\frac12\|\sqrt{\rho_h}\uh^+\|_{X,h,K}^2
  &=\frac12\sum_{\ell=1}^p
    \left[1-\mathcal D_{\ell,K}^2\right]
    \|\boldsymbol w_{\ell,K}\|_{\rho,K}^2\\
  &=(\rho_h(e_h^+-e_h^-),1)_{V,h,K}\ge0.
\end{aligned}
\end{equation}
Hence one velocity-OE postprocessing map preserves the collocated element
fluid energy and cannot reduce a positive nodal internal energy.
\end{proposition}

\begin{proof}
The weighted constant projection \(P_{0,K}^{\rho}\) has the same element
momentum as the input, while every increment with \(\ell\ge1\) is orthogonal
to constants in \eqref{eq:rho-weighted-inner-product}. Orthogonality of the
nested projection increments gives the kinetic-energy difference in
\eqref{eq:magnetic-oe-energy-balance}. Integrating the nodal increment in
\eqref{eq:magnetic-oe-map} gives the same sum. Since
\(0\le\mathcal D_{\ell,K}\le1\), every term is nonnegative.
\end{proof}

\paragraph{Artificial-resistivity coefficient.}
Artificial resistivity is embedded directly in the semidiscrete system
\eqref{eq:magnetic-dae}, rather than applied as a postprocessing map. Its
elementwise coefficient is determined by a degree-zero directional jump
sensor for the magnetic field. For \(I\in\{x,y\}\), let
\(\mathcal T_x=\{y,z\}\) and \(\mathcal T_y=\{x,z\}\) denote the components
tangential to faces normal to direction \(I\), and set
\begin{equation}
\label{eq:magnetic-ar-sensor}
  \Delta_I^B
  =\max_{d\in\mathcal T_I}\Delta_\Omega(B_{d,h}),
  \qquad
  \sigma^B_{I,K}
  =\max_{d\in\mathcal T_I}
   \frac{J_{0,I,K}(B_{d,h})}{2\Delta_I^B}.
\end{equation}
Here \(\Delta_\Omega\) and \(J_{0,I,K}\) are defined in
\eqref{eq:euler-oe-amplitude} and \eqref{eq:euler-oe-jump}, respectively;
a direction with only roundoff-level magnetic variation contributes zero.
Only tangential jumps enter because the normal trace of the in-plane
Raviart--Thomas field is single valued.

The sensor response, directional speed, and artificial-resistivity
coefficient are
\begin{equation}
\label{eq:magnetic-ar-coefficient}
\begin{aligned}
  R_{I,K}&=\min\{C_s\sigma^B_{I,K},1\},\\
  a_{I,K}&=\max_{x_q\in K}\left(
    |u_{I,h}(x_q)|+\frac{|\Bh(x_q)|}{\sqrt{\rho_h(x_q)}}\right),\\
  \eta_K&=C_\eta\max_{I\in\{x,y\}}
    \left\{\frac{h_{I,K}}{2p+1}\,a_{I,K}R_{I,K}\right\}.
\end{aligned}
\end{equation}
Thus \(\eta_h|_K=\eta_K\ge0\), with \(C_\eta\ge0\) setting the maximum
dissipation and \(C_s\ge0\) controlling the sensor response. We use
\(C_s=5\) throughout and \(C_\eta=0.2\) in the nonsmooth tests;
\(C_\eta=0\) disables artificial resistivity. The coefficient is computed
from the state entering a magnetic substep and held fixed throughout that
substep.

By Theorem~\ref{thm:energy}, artificial resistivity converts resolved magnetic
energy into internal energy without changing the semidiscrete total energy.
It also leaves the magnetic update in curl form, so
Theorem~\ref{thm:magnetic-divergence} remains valid. We do not apply OE
directly to \(\Bh\), since a generic element-local filter need not preserve
the compatible divergence-free subspace.

\paragraph{Entropy stability.}
For positive nodal internal energy, define the physical specific entropy and
its DG--GLL integral by
\[
  s_{h,q}=\log p_{h,q}-\gamma\log\rho_{h,q},
  \qquad
  \mathcal S_h=(\rho_hs_h,1)_{V,h},
\]
and let
\(\mathcal E_{\mathrm{ent},h}=-\mathcal S_h/(\gamma-1)\) be the
corresponding convex mathematical entropy.

\begin{theorem}[Entropy stability of the magnetic-stage stabilizations]
\label{thm:magnetic-stabilization-entropy}
Suppose that the nodal density and internal energy satisfy
\(\rho_{h,q}>0\) and \(e_{h,q}>0\). For fixed \(\rho_h\) and
\(\eta_h|_K=\eta_K\ge0\), the semidiscrete Ohmic source satisfies
\begin{equation}
\label{eq:magnetic-ar-entropy}
  \frac{d\mathcal S_h}{dt}
  =\sum_{K\in\mathcal T_h}\sum_{x_q\in K}
    M_q^K\frac{\eta_K|\Jh(x_q)|^2}{e_h(x_q)}\ge0.
\end{equation}
Moreover, one velocity-OE map applied to a positive nodal state satisfies
\begin{equation}
\label{eq:magnetic-oe-entropy}
  \mathcal S_h^+-\mathcal S_h^-
  =\sum_{K\in\mathcal T_h}\sum_{x_q\in K}
    M_q^K\rho_{h,q}\log\!\left(
      \frac{e_{h,q}^+}{e_{h,q}^-}\right)\ge0.
\end{equation}
Consequently, both stabilization mechanisms increase the discrete physical
entropy, or equivalently do not increase \(\mathcal E_{\mathrm{ent},h}\).
\end{theorem}

\begin{proof}
Because the density is fixed during the magnetic substep and
\(p_h=(\gamma-1)\rho_he_h\), the nodal Ohmic heating equation gives
\[
  \rho_{h,q}\frac{d s_{h,q}}{dt}
  =\frac{\rho_{h,q}}{e_{h,q}}\frac{d e_{h,q}}{dt}
  =\frac{\eta_K|\Jh(x_q)|^2}{e_{h,q}}.
\]
Multiplication by the positive GLL weights and summation yield
\eqref{eq:magnetic-ar-entropy}. For the velocity-OE map, fixed density gives
\(s_{h,q}^+-s_{h,q}^-=\log(e_{h,q}^+/e_{h,q}^-)\), while
Proposition~\ref{prop:magnetic-oe-energy} ensures
\(e_{h,q}^+\ge e_{h,q}^->0\). Summation therefore yields
\eqref{eq:magnetic-oe-entropy}.
\end{proof}

\subsubsection{SSP-RK(3,3) update}

Let \(Y_h=(e_h,\uh,\Bh)\), and let
\(\mathcal L_M(Y_h;\rho_h,\eta_h)\) denote the mass-inverted differential
operator obtained after eliminating \(\Jh\) and \(\Eh\) from
\eqref{eq:magnetic-dae}. The density and artificial-resistivity coefficient
are fixed parameters of this operator during one magnetic substep. Define the
stage map
\begin{equation}
\label{eq:magnetic-stage-map}
  \mathcal P_M^\tau(e_h,\uh,\Bh)=(e_h^+,\uh^+,\Bh),
\end{equation}
where \((e_h^+,\uh^+)\) is given elementwise by
\eqref{eq:magnetic-oe-map}; when velocity OE is disabled,
\(\mathcal P_M^\tau\) is the identity. The stabilized SSP-RK(3,3) update is
\begin{subequations}
\label{eq:magnetic-ssprk33}
\begin{align}
  Y^{(1)}
  &=\mathcal P_M^\tau\!\left(
    Y^n+\tau\mathcal L_M(Y^n;\rho_h,\eta_h)\right),\\
  Y^{(2)}
  &=\mathcal P_M^\tau\!\left(
    \frac34Y^n+\frac14\left[Y^{(1)}
    +\tau\mathcal L_M(Y^{(1)};\rho_h,\eta_h)\right]\right),\\
  Y^{n+1}
  &=\mathcal P_M^\tau\!\left(
    \frac13Y^n+\frac23\left[Y^{(2)}
    +\tau\mathcal L_M(Y^{(2)};\rho_h,\eta_h)\right]\right).
\end{align}
\end{subequations}
At each residual evaluation, \(\Jh\) and \(\Eh\) are reconstructed from the
two algebraic equations in \eqref{eq:magnetic-dae}. Starting from positive
\(e_h^n\), every Ohmic forward-Euler increment is nonnegative, and each OE map
adds a nonnegative nodal internal-energy increment. Moreover, each raw SSP-RK
stage is a convex combination of \(e_h^n\) and a preceding stage followed by
such an increment. Induction therefore gives
\(e_{h,q}^{(i)}\ge e_{h,q}^n>0\) at every stage and
\begin{equation}
\label{eq:magnetic-entropy-stability}
  \mathcal E_{\mathrm{ent},h}^{n+1}
  \le\mathcal E_{\mathrm{ent},h}^n.
\end{equation}
Thus the stabilized magnetic--velocity update is entropy stable and preserves
positive nodal internal energy.
Since \(\mathcal P_M^\tau\) leaves \(\Bh\) unchanged, the in-plane
divergence-free property is also preserved stage by stage. We note, however,
that the exact energy identity in Theorem~\ref{thm:energy} is semidiscrete;
the explicit Runge--Kutta update does not in general preserve this quadratic
energy exactly.

Algorithm~\ref{alg:magnetic-step} summarizes the magnetic-velocity substep
defined in this section.

\begin{algorithm}[H]
\caption{Magnetic-velocity substep \(\Phi_M^\tau\)}
\label{alg:magnetic-step}
\begin{algorithmic}[1]
\Require \((\rho_h,e_h,\uh,\Bh)^n\) and time step \(\tau\).
\State Freeze \(\rho_h\). Set \(\eta_h=0\), or, when artificial resistivity
       is enabled, compute \(\eta_h\) from
       \eqref{eq:magnetic-ar-coefficient}. Hold it fixed over the substep.
\State At every SSP-RK residual evaluation, reconstruct \(\Jh\) and \(\Eh\)
       from \eqref{eq:mag-current-proj} and
       \eqref{eq:mag-electric-proj}, respectively, and form
       \(\mathcal L_M\).
\State Advance \(Y_h=(e_h,\uh,\Bh)\) with
       \eqref{eq:magnetic-ssprk33}, applying \(\mathcal P_M^\tau\) to every
       provisional stage.
\Ensure \((\rho_h,e_h,\uh,\Bh)^{\mathrm{out}}\), with \(\rho_h\) unchanged
        and nonnegative Ohmic and velocity-OE heating included in \(e_h\).
\end{algorithmic}
\end{algorithm}

\subsection{Strang splitting and the fully discrete algorithm}
\label{sec:algorithm}

Let \(\Phi_H^\tau\) denote the hydrodynamic Euler map in
Algorithm~\ref{alg:hydro-step}, and let \(\Phi_M^\tau\) denote the
magnetic--velocity map in Algorithm~\ref{alg:magnetic-step}. The fully
discrete solver advances each time step with the symmetric Strang composition
\cite{strang1968construction}
\begin{equation}
\label{eq:strang-composition}
  \mathcal S_{HMH}^{\Delta t}
  =\Phi_H^{\Delta t/2}\circ\Phi_M^{\Delta t}
    \circ\Phi_H^{\Delta t/2}.
\end{equation}
For smooth solutions, this composition is second-order in time.

\subsubsection{Local variable exchange}

The hydrodynamic step evolves the conservative Euler variables
\[
  U_h=(\rho_h,\boldsymbol m_h,E_{\mathrm{mech},h})^T\in [V_h^p]^5,
\]
whereas the magnetic step evolves \((e_h,\uh,\Bh)\) while freezing
\(\rho_h\). The scalar fields \(\rho_h\) and \(e_h\) are both
represented in the DG tensor-product space \(V_h^p\). Therefore the exchange
between the two sets of variables is performed pointwise at the GLL
collocation nodes. On each element \(K\), for \(0\le i,j\le p\),
\[
  \boldsymbol u_{ij}^K
  =
  \frac{\boldsymbol m_{ij}^K}{\rho_{ij}^K},
  \qquad
  e_{ij}^K
  =
  \frac{E_{\mathrm{mech},ij}^K
  -\frac12\rho_{ij}^K|\boldsymbol u_{ij}^K|^2}{\rho_{ij}^K}.
\]
After a magnetic update, the conservative variables are recovered at the same
nodes by
\[
  \boldsymbol m_{ij}^K=\rho_{ij}^K\boldsymbol u_{ij}^K,
  \qquad
  E_{\mathrm{mech},ij}^K
  =
  \rho_{ij}^K e_{ij}^K
  +\frac12\rho_{ij}^K|\boldsymbol u_{ij}^K|^2 .
\]
This nodal collocation exchange requires no global solve and no inter-element
communication.

\subsubsection{One Strang step}

Algorithm~\ref{alg:strang-step} summarizes one fully discrete step.

\begin{algorithm}[H]
\caption{Strang split MHD step}
\label{alg:strang-step}
\begin{algorithmic}[1]
\Require \((\rho_h,\boldsymbol m_h,E_{\mathrm{mech},h},\Bh)^n\) and time step
         \(\Delta t\).
\State Apply the hydrodynamic half-step
       \(\Phi_H^{\Delta t/2}\) using Algorithm~\ref{alg:hydro-step}.
\State Convert the updated conservative variables to
       \((\rho_h,e_h,\uh,\Bh)\) by the nodal formulas above.
\State Apply the magnetic full-step
       \(\Phi_M^{\Delta t}\) using Algorithm~\ref{alg:magnetic-step}.
\State Recover \((\rho_h,\boldsymbol m_h,E_{\mathrm{mech},h})\) at the GLL
       nodes.
\State Apply the second hydrodynamic half-step
       \(\Phi_H^{\Delta t/2}\) using Algorithm~\ref{alg:hydro-step}.
\Ensure \((\rho_h,\boldsymbol m_h,E_{\mathrm{mech},h},\Bh)^{n+1}\).
\end{algorithmic}
\end{algorithm}

The complete method is fully explicit and globally mass conservative. The
hydrodynamic positivity limiter maintains positive density and pressure at the
GLL nodes whenever the cell averages are admissible; otherwise, the step is
retried with a smaller time step. For an initially compatible magnetic field,
the curl update preserves the global \(H(\mathrm{div})\) divergence-free
constraint, while the stabilized magnetic stage satisfies the entropy
inequality \eqref{eq:magnetic-entropy-stability}. The method does not, however,
exactly conserve total energy or momentum at the fully discrete level: SSP-RK
does not generally preserve the quadratic magnetic-stage energy, and the
discrete Lorentz-force update \eqref{eq:semi-vel} is not exactly momentum
conservative. Likewise, the semidiscrete Euler entropy estimate in
Theorem~\ref{thm:euler-dgsem-entropy} does not by itself imply a fully discrete
entropy inequality after the Euler stage stabilization.

\section{Numerical experiments}
\label{sec:numerics}

The numerical implementation is built on the MFEM finite element library
\cite{anderson2021mfem}. All full-system experiments below use the Strang
splitting in Algorithm~\ref{alg:strang-step}, whereas the reduced
Alfv\'en-wave study advances only the magnetic--velocity subproblem with
Algorithm~\ref{alg:magnetic-step}. For 2D2V configurations, in which
\(u_z=B_z=0\) and the in-plane components of \(\boldsymbol E\) and
\(\boldsymbol J\) vanish, we solve the corresponding reduced system without
these identically zero variables.

Throughout this section, \(p\) denotes the DGSEM degree of the Euler variables
and \(m\) the base degree of the compatible N\'ed\'elec--Raviart--Thomas
complex. The convergence studies compare the equal-order choice \(m=p\) with
the reduced magnetic degree \(m=p-1\). All subsequent nonsmooth tests use
\(m=p=2\), activate both OE stages and the positivity limiter, and compare
otherwise identical calculations without and with artificial resistivity
(A.R.). In every \(p=2\) simulation with OE, we set \(s_H=s_M=0.02\). When
A.R. is active, its coefficient is computed from
\eqref{eq:magnetic-ar-coefficient} with \(C_s=5\) and \(C_\eta=0.2\). Any
maximum A.R. coefficient quoted below is the spatial maximum over the full
mesh of the coefficient used in the final magnetic substep.

\subsection{Reduced Alfv\'en-wave magnetic-stage convergence}
\label{sec:magnetic-stage-spatial-convergence}

We first isolate the magnetic--velocity stage from the Euler update and test
the spatial accuracy of the compatible discretization. For constant density,
the reduced system is
\begin{subequations}
\label{eq:mag-stage-continuous}
\begin{align}
 \rho_0\partial_t\boldsymbol u+\boldsymbol B\times\boldsymbol J&=0,
 & \boldsymbol J&=\nabla\times\boldsymbol B, \\
 \partial_t\boldsymbol B+\nabla\times\boldsymbol E&=0,
 & \boldsymbol E&=-\boldsymbol u\times\boldsymbol B.
\end{align}
\end{subequations}
The periodic domain is \(\Omega=[0,2\pi]^2\), and a circularly polarized
Alfv\'en wave travels in the diagonal direction
\[
 \boldsymbol n=\frac{1}{\sqrt{2}}(1,1,0)^T,
 \qquad
 \boldsymbol t=\frac{1}{\sqrt{2}}(-1,1,0)^T.
\]
Let \(k=\sqrt{2}\), \(\omega=kB_0/\sqrt{\rho_0}\), and
\(\theta=x+y-\omega t\). With
\(\rho_0=p_0=B_0=1\), \(\gamma=5/3\), and perturbation amplitude
\(\varepsilon=0.1\), the exact magnetic field is
\begin{equation}
 \boldsymbol B
 =B_0\boldsymbol n
 +\varepsilon
  \left(\boldsymbol t\cos\theta+\boldsymbol e_z\sin\theta\right),
 \label{eq:mag-stage-alfven-B}
\end{equation}
and the exact velocity is
\begin{equation}
 \boldsymbol u
 =-\frac{\varepsilon}{\sqrt{\rho_0}}
 \left(\boldsymbol t\cos\theta+\boldsymbol e_z\sin\theta\right).
 \label{eq:mag-stage-alfven-u}
\end{equation}

For every velocity degree \(p=1,2,3,4\), we compare the two adjacent
compatible magnetic degrees \(m=p-1\) and \(m=p\). To expose the parity
dependence, we organize these choices into the even family
\[
 m_{\rm even}=2\left\lfloor\frac{p}{2}\right\rfloor
 \quad\bigl((p,m)=(1,0),(2,2),(3,2),(4,4)\bigr)
\]
and the odd family
\[
 m_{\rm odd}=2\left\lfloor\frac{p-1}{2}\right\rfloor+1
 \quad\bigl((p,m)=(1,1),(2,1),(3,3),(4,3)\bigr).
\]

Uniform \(N\times N\) meshes with \(N=16,32,64,128\) are evolved to
\(t_f=1\) with the fixed time step \(\Delta t=10^{-4}\), which makes the
temporal error negligible relative to the spatial error. OE, positivity
limiting, and A.R. are disabled. For
\(q\in\{\boldsymbol u,\boldsymbol B,\boldsymbol E,\boldsymbol J\}\), we define
\[
 e_q=\lVert q_h-q\rVert_{L^2(\Omega)}.
\]
The reported vector errors combine the in-plane and out-of-plane components;
for example,
\[
 e_B=\left(e_{B_{xy}}^2+e_{B_z}^2\right)^{1/2},
\]
and analogously for \(e_E\) and \(e_J\). For each successive refinement, we
report the rate
\[
 r_q(2N)=\log_2\!\left(\frac{e_q(N)}{e_q(2N)}\right).
\]

\begin{table}[t]
\centering
\caption{Spatial convergence of the reduced Alfv\'en-wave magnetic stage,
grouped by the parity of the magnetic degree $m$. The even and odd families
use $m=2\lfloor p/2\rfloor$ and
$m=2\lfloor(p-1)/2\rfloor+1$, respectively. Each entry gives the $L^2$ error
followed by the successive-refinement rate in parentheses.}
\label{tab:mag-stage-spatial-parity}
\scriptsize
\setlength{\tabcolsep}{3.0pt}
\begin{tabular}{@{}ccrcccc@{}}
\toprule
$p$ & $m$ & $N$ & $e_u$ & $e_B$ & $e_E$ & $e_J$ \\
\midrule
\multicolumn{7}{@{}l}{\textit{Even magnetic degree}} \\
1 & 0 & 16 & $5.367\mathrm{e}{-2}\;(--)$ & $8.803\mathrm{e}{-2}\;(--)$ & $5.782\mathrm{e}{-2}\;(--)$ & $8.042\mathrm{e}{-2}\;(--)$ \\
1 & 0 & 32 & $2.467\mathrm{e}{-2}\;(1.12)$ & $4.372\mathrm{e}{-2}\;(1.01)$ & $2.620\mathrm{e}{-2}\;(1.14)$ & $3.686\mathrm{e}{-2}\;(1.13)$ \\
1 & 0 & 64 & $1.204\mathrm{e}{-2}\;(1.04)$ & $2.182\mathrm{e}{-2}\;(1.00)$ & $1.272\mathrm{e}{-2}\;(1.04)$ & $1.797\mathrm{e}{-2}\;(1.04)$ \\
1 & 0 & 128 & $5.981\mathrm{e}{-3}\;(1.01)$ & $1.091\mathrm{e}{-2}\;(1.00)$ & $6.312\mathrm{e}{-3}\;(1.01)$ & $8.923\mathrm{e}{-3}\;(1.01)$ \\
\addlinespace[1.5pt]
2 & 2 & 16 & $6.337\mathrm{e}{-4}\;(--)$ & $9.094\mathrm{e}{-4}\;(--)$ & $4.242\mathrm{e}{-4}\;(--)$ & $1.654\mathrm{e}{-2}\;(--)$ \\
2 & 2 & 32 & $6.800\mathrm{e}{-5}\;(3.22)$ & $1.334\mathrm{e}{-4}\;(2.77)$ & $5.287\mathrm{e}{-5}\;(3.00)$ & $5.030\mathrm{e}{-3}\;(1.72)$ \\
2 & 2 & 64 & $6.690\mathrm{e}{-6}\;(3.35)$ & $1.362\mathrm{e}{-5}\;(3.29)$ & $5.464\mathrm{e}{-6}\;(3.27)$ & $9.559\mathrm{e}{-4}\;(2.40)$ \\
2 & 2 & 128 & $1.215\mathrm{e}{-6}\;(2.46)$ & $2.006\mathrm{e}{-6}\;(2.76)$ & $8.721\mathrm{e}{-7}\;(2.65)$ & $3.057\mathrm{e}{-4}\;(1.64)$ \\
\addlinespace[1.5pt]
3 & 2 & 16 & $8.536\mathrm{e}{-5}\;(--)$ & $1.486\mathrm{e}{-4}\;(--)$ & $9.856\mathrm{e}{-5}\;(--)$ & $2.819\mathrm{e}{-4}\;(--)$ \\
3 & 2 & 32 & $1.070\mathrm{e}{-5}\;(3.00)$ & $1.847\mathrm{e}{-5}\;(3.01)$ & $1.231\mathrm{e}{-5}\;(3.00)$ & $4.922\mathrm{e}{-5}\;(2.52)$ \\
3 & 2 & 64 & $1.340\mathrm{e}{-6}\;(3.00)$ & $2.300\mathrm{e}{-6}\;(3.01)$ & $1.537\mathrm{e}{-6}\;(3.00)$ & $8.551\mathrm{e}{-6}\;(2.53)$ \\
3 & 2 & 128 & $1.659\mathrm{e}{-7}\;(3.01)$ & $2.872\mathrm{e}{-7}\;(3.00)$ & $1.908\mathrm{e}{-7}\;(3.01)$ & $1.691\mathrm{e}{-6}\;(2.34)$ \\
\addlinespace[1.5pt]
4 & 4 & 16 & $3.417\mathrm{e}{-7}\;(--)$ & $4.384\mathrm{e}{-7}\;(--)$ & $2.631\mathrm{e}{-7}\;(--)$ & $1.482\mathrm{e}{-5}\;(--)$ \\
4 & 4 & 32 & $9.945\mathrm{e}{-9}\;(5.10)$ & $1.123\mathrm{e}{-8}\;(5.29)$ & $7.422\mathrm{e}{-9}\;(5.15)$ & $7.646\mathrm{e}{-7}\;(4.28)$ \\
4 & 4 & 64 & $2.977\mathrm{e}{-10}\;(5.06)$ & $3.614\mathrm{e}{-10}\;(4.96)$ & $2.115\mathrm{e}{-10}\;(5.13)$ & $4.449\mathrm{e}{-8}\;(4.10)$ \\
4 & 4 & 128 & $9.345\mathrm{e}{-12}\;(4.99)$ & $1.018\mathrm{e}{-11}\;(5.15)$ & $6.704\mathrm{e}{-12}\;(4.98)$ & $2.682\mathrm{e}{-9}\;(4.05)$ \\
\midrule
\multicolumn{7}{@{}l}{\textit{Odd magnetic degree}} \\
1 & 1 & 16 & $7.512\mathrm{e}{-2}\;(--)$ & $7.792\mathrm{e}{-2}\;(--)$ & $4.259\mathrm{e}{-2}\;(--)$ & $6.176\mathrm{e}{-1}\;(--)$ \\
1 & 1 & 32 & $3.770\mathrm{e}{-2}\;(0.99)$ & $3.914\mathrm{e}{-2}\;(0.99)$ & $2.109\mathrm{e}{-2}\;(1.01)$ & $6.184\mathrm{e}{-1}\;(-0.00)$ \\
1 & 1 & 64 & $1.887\mathrm{e}{-2}\;(1.00)$ & $1.959\mathrm{e}{-2}\;(1.00)$ & $1.052\mathrm{e}{-2}\;(1.00)$ & $6.186\mathrm{e}{-1}\;(-0.00)$ \\
1 & 1 & 128 & $9.436\mathrm{e}{-3}\;(1.00)$ & $9.799\mathrm{e}{-3}\;(1.00)$ & $5.258\mathrm{e}{-3}\;(1.00)$ & $6.186\mathrm{e}{-1}\;(-0.00)$ \\
\addlinespace[1.5pt]
2 & 1 & 16 & $5.901\mathrm{e}{-4}\;(--)$ & $4.549\mathrm{e}{-3}\;(--)$ & $2.583\mathrm{e}{-3}\;(--)$ & $8.284\mathrm{e}{-3}\;(--)$ \\
2 & 1 & 32 & $1.224\mathrm{e}{-4}\;(2.27)$ & $1.116\mathrm{e}{-3}\;(2.03)$ & $6.420\mathrm{e}{-4}\;(2.01)$ & $2.169\mathrm{e}{-3}\;(1.93)$ \\
2 & 1 & 64 & $2.663\mathrm{e}{-5}\;(2.20)$ & $2.769\mathrm{e}{-4}\;(2.01)$ & $1.601\mathrm{e}{-4}\;(2.00)$ & $5.454\mathrm{e}{-4}\;(1.99)$ \\
2 & 1 & 128 & $6.140\mathrm{e}{-6}\;(2.12)$ & $6.923\mathrm{e}{-5}\;(2.00)$ & $3.999\mathrm{e}{-5}\;(2.00)$ & $2.640\mathrm{e}{-4}\;(1.05)$ \\
\addlinespace[1.5pt]
3 & 3 & 16 & $7.826\mathrm{e}{-5}\;(--)$ & $1.028\mathrm{e}{-4}\;(--)$ & $4.782\mathrm{e}{-5}\;(--)$ & $1.856\mathrm{e}{-3}\;(--)$ \\
3 & 3 & 32 & $9.723\mathrm{e}{-6}\;(3.01)$ & $1.287\mathrm{e}{-5}\;(3.00)$ & $5.848\mathrm{e}{-6}\;(3.03)$ & $4.636\mathrm{e}{-4}\;(2.00)$ \\
3 & 3 & 64 & $1.213\mathrm{e}{-6}\;(3.00)$ & $1.610\mathrm{e}{-6}\;(3.00)$ & $7.322\mathrm{e}{-7}\;(3.00)$ & $1.158\mathrm{e}{-4}\;(2.00)$ \\
3 & 3 & 128 & $1.514\mathrm{e}{-7}\;(3.00)$ & $2.013\mathrm{e}{-7}\;(3.00)$ & $9.156\mathrm{e}{-8}\;(3.00)$ & $2.893\mathrm{e}{-5}\;(2.00)$ \\
\addlinespace[1.5pt]
4 & 3 & 16 & $1.271\mathrm{e}{-6}\;(--)$ & $3.662\mathrm{e}{-6}\;(--)$ & $2.135\mathrm{e}{-6}\;(--)$ & $9.514\mathrm{e}{-6}\;(--)$ \\
4 & 3 & 32 & $7.306\mathrm{e}{-8}\;(4.12)$ & $2.276\mathrm{e}{-7}\;(4.01)$ & $1.319\mathrm{e}{-7}\;(4.02)$ & $7.156\mathrm{e}{-7}\;(3.73)$ \\
4 & 3 & 64 & $4.467\mathrm{e}{-9}\;(4.03)$ & $1.420\mathrm{e}{-8}\;(4.00)$ & $8.221\mathrm{e}{-9}\;(4.00)$ & $7.365\mathrm{e}{-8}\;(3.28)$ \\
4 & 3 & 128 & $2.769\mathrm{e}{-10}\;(4.01)$ & $8.880\mathrm{e}{-10}\;(4.00)$ & $5.134\mathrm{e}{-10}\;(4.00)$ & $1.022\mathrm{e}{-8}\;(2.85)$ \\
\bottomrule
\end{tabular}
\end{table}

Table~\ref{tab:mag-stage-spatial-parity} shows a clear parity dependence.
For the even-\(m\) family, the magnetic and electric fields converge at
approximately order \(m+1\), while the current generally converges at order
\(m\). The low-order pair \((p,m)=(1,0)\) is a favorable exception: its
current remains first-order accurate. The velocity converges at approximately
order \(p+1\) for even \(p\) and order \(p\) for odd \(p\). For the odd-\(m\)
family, the velocity, magnetic field, and electric field converge at
approximately order \(p\), whereas the current converges at order \(p-1\).
The one-order loss for \(J_h\) relative to \(B_h\) is consistent with
the differentiation in \(\boldsymbol J_h=\nabla\times\boldsymbol B_h\).

The most striking contrast occurs at \(p=1\). For \((p,m)=(1,0)\), the
current converges at first order and reaches \(e_J=8.92\times10^{-3}\) on the
finest mesh. For \((p,m)=(1,1)\), however, the current error stagnates at
\(e_J\approx6.19\times10^{-1}\), nearly 70 times larger. The primary
variables remain first-order accurate for both choices, which localizes the
failure to the current reconstruction. At higher degree, \((2,2)\) and
\((4,4)\) recover approximately third- and fifth-order primary fields,
respectively, together with the expected lower orders for the current. Both
\(p=3\) choices exhibit robust third-order
primary-field convergence, although \((3,3)\) gives smaller absolute primary
errors. Overall, the even magnetic degrees provide the more reliable
high-order behavior.

\subsection{Advected MHD vortex}
\label{sec:balsara-vortex-convergence}

We next test the fully coupled system using the isodensity MHD vortex of
Balsara~\cite{balsara2004second}, for which both split subflows are active.
We take \(\gamma=5/3\). Computations are performed on the periodic box
\(\Omega_c=[-10,10]^2\), and errors are evaluated over the central region
\(\Omega_e=[-5,5]^2\). Let
\(\overline{\boldsymbol u}=(1,1,0)^T\),
\(\boldsymbol\xi=(x-t,y-t)^T\), \(r^2=\lvert\boldsymbol\xi\rvert^2\), and
\(g=\exp((1-r^2)/2)\), with the displacement interpreted periodically on
\(\Omega_c\). In the normalized magnetic units used here, the exact solution
is
\begin{align*}
 \rho &= 1, &
 \boldsymbol u &= \overline{\boldsymbol u}
   +\frac{g}{2\pi}(-\xi_y,\xi_x,0)^T,\\
 p &= 1-\frac{r^2g^2}{8\pi^2}, &
 \boldsymbol B &= \frac{g}{2\pi}(-\xi_y,\xi_x,0)^T.
\end{align*}
Thus, the density is constant, whereas the nonuniform pressure, velocity, and
magnetic profiles advect with the background velocity without changing shape.
We initialize the in-plane magnetic field as the discrete curl of the
corresponding scalar potential.

\subsubsection{Spatial convergence}
\label{sec:balsara-vortex-spatial-convergence}

For the spatial study, we use the same even- and odd-\(m\) families as in the
reduced Alfv\'en test. Uniform meshes with \(N=16,32,64,128\) are evolved to
\(t_f=1\) with the fixed time step \(\Delta t=10^{-4}\); by this time, the
background flow has translated the vortex by \((1,1)\). As in the reduced
test, the time step suppresses temporal-error contamination. We report
the \(L^2\) errors in density, velocity, pressure, magnetic field,
out-of-plane electric field, and out-of-plane current. OE, positivity limiting, and
artificial resistivity are disabled.

\begin{table}[t]
\centering
\caption{Spatial convergence of the advected MHD vortex, grouped by the
parity of the magnetic degree $m$. The even and odd families use
$m=2\lfloor p/2\rfloor$ and $m=2\lfloor(p-1)/2\rfloor+1$, respectively. Each
entry gives the $L^2(\Omega_e)$ error followed by the successive-refinement
rate in parentheses.}
\label{tab:balsara-vortex-spatial-parity}
\scriptsize
\setlength{\tabcolsep}{1.6pt}
\resizebox{\linewidth}{!}{%
\begin{tabular}{@{}ccrcccccc@{}}
\toprule
$p$ & $m$ & $N$ & $e_\rho$ & $e_u$ & $e_p$ & $e_B$ & $e_{E_z}$ & $e_{J_z}$ \\
\midrule
\multicolumn{9}{@{}l}{\textit{Even magnetic degree}} \\
1 & 0 & 16 & $2.780\mathrm{e}{-2}\;(--)$ & $2.178\mathrm{e}{-1}\;(--)$ & $5.042\mathrm{e}{-2}\;(--)$ & $3.163\mathrm{e}{-1}\;(--)$ & $2.919\mathrm{e}{-1}\;(--)$ & $5.110\mathrm{e}{-1}\;(--)$ \\
1 & 0 & 32 & $1.140\mathrm{e}{-2}\;(1.29)$ & $9.312\mathrm{e}{-2}\;(1.23)$ & $2.116\mathrm{e}{-2}\;(1.25)$ & $1.410\mathrm{e}{-1}\;(1.17)$ & $1.249\mathrm{e}{-1}\;(1.22)$ & $2.372\mathrm{e}{-1}\;(1.11)$ \\
1 & 0 & 64 & $3.320\mathrm{e}{-3}\;(1.78)$ & $2.923\mathrm{e}{-2}\;(1.67)$ & $6.060\mathrm{e}{-3}\;(1.80)$ & $5.844\mathrm{e}{-2}\;(1.27)$ & $3.720\mathrm{e}{-2}\;(1.75)$ & $7.288\mathrm{e}{-2}\;(1.70)$ \\
1 & 0 & 128 & $8.388\mathrm{e}{-4}\;(1.98)$ & $7.840\mathrm{e}{-3}\;(1.90)$ & $1.475\mathrm{e}{-3}\;(2.04)$ & $2.669\mathrm{e}{-2}\;(1.13)$ & $9.745\mathrm{e}{-3}\;(1.93)$ & $1.922\mathrm{e}{-2}\;(1.92)$ \\
\addlinespace[1.5pt]
2 & 2 & 16 & $8.185\mathrm{e}{-3}\;(--)$ & $3.666\mathrm{e}{-2}\;(--)$ & $1.250\mathrm{e}{-2}\;(--)$ & $8.822\mathrm{e}{-2}\;(--)$ & $1.610\mathrm{e}{-2}\;(--)$ & $4.272\mathrm{e}{-1}\;(--)$ \\
2 & 2 & 32 & $1.436\mathrm{e}{-3}\;(2.51)$ & $4.671\mathrm{e}{-3}\;(2.97)$ & $2.385\mathrm{e}{-3}\;(2.39)$ & $7.290\mathrm{e}{-3}\;(3.60)$ & $2.065\mathrm{e}{-3}\;(2.96)$ & $7.441\mathrm{e}{-2}\;(2.52)$ \\
2 & 2 & 64 & $2.638\mathrm{e}{-4}\;(2.44)$ & $8.639\mathrm{e}{-4}\;(2.43)$ & $4.358\mathrm{e}{-4}\;(2.45)$ & $4.222\mathrm{e}{-4}\;(4.11)$ & $1.708\mathrm{e}{-4}\;(3.60)$ & $8.768\mathrm{e}{-3}\;(3.09)$ \\
2 & 2 & 128 & $4.776\mathrm{e}{-5}\;(2.47)$ & $1.382\mathrm{e}{-4}\;(2.64)$ & $7.881\mathrm{e}{-5}\;(2.47)$ & $5.025\mathrm{e}{-5}\;(3.07)$ & $1.835\mathrm{e}{-5}\;(3.22)$ & $2.150\mathrm{e}{-3}\;(2.03)$ \\
\addlinespace[1.5pt]
3 & 2 & 16 & $7.804\mathrm{e}{-3}\;(--)$ & $1.449\mathrm{e}{-2}\;(--)$ & $1.305\mathrm{e}{-2}\;(--)$ & $8.328\mathrm{e}{-2}\;(--)$ & $1.510\mathrm{e}{-2}\;(--)$ & $4.032\mathrm{e}{-1}\;(--)$ \\
3 & 2 & 32 & $7.359\mathrm{e}{-4}\;(3.41)$ & $1.367\mathrm{e}{-3}\;(3.41)$ & $1.246\mathrm{e}{-3}\;(3.39)$ & $6.354\mathrm{e}{-3}\;(3.71)$ & $1.704\mathrm{e}{-3}\;(3.15)$ & $6.418\mathrm{e}{-2}\;(2.65)$ \\
3 & 2 & 64 & $6.872\mathrm{e}{-5}\;(3.42)$ & $1.165\mathrm{e}{-4}\;(3.55)$ & $1.150\mathrm{e}{-4}\;(3.44)$ & $3.895\mathrm{e}{-4}\;(4.03)$ & $1.242\mathrm{e}{-4}\;(3.78)$ & $8.022\mathrm{e}{-3}\;(3.00)$ \\
3 & 2 & 128 & $8.307\mathrm{e}{-6}\;(3.05)$ & $1.157\mathrm{e}{-5}\;(3.33)$ & $1.374\mathrm{e}{-5}\;(3.07)$ & $3.893\mathrm{e}{-5}\;(3.32)$ & $1.074\mathrm{e}{-5}\;(3.53)$ & $1.599\mathrm{e}{-3}\;(2.33)$ \\
\addlinespace[1.5pt]
4 & 4 & 16 & $4.438\mathrm{e}{-4}\;(--)$ & $9.424\mathrm{e}{-4}\;(--)$ & $6.053\mathrm{e}{-4}\;(--)$ & $3.994\mathrm{e}{-3}\;(--)$ & $2.880\mathrm{e}{-4}\;(--)$ & $4.626\mathrm{e}{-2}\;(--)$ \\
4 & 4 & 32 & $1.885\mathrm{e}{-5}\;(4.56)$ & $3.567\mathrm{e}{-5}\;(4.72)$ & $3.099\mathrm{e}{-5}\;(4.29)$ & $5.122\mathrm{e}{-5}\;(6.29)$ & $1.226\mathrm{e}{-5}\;(4.55)$ & $1.213\mathrm{e}{-3}\;(5.25)$ \\
4 & 4 & 64 & $7.912\mathrm{e}{-7}\;(4.57)$ & $1.464\mathrm{e}{-6}\;(4.61)$ & $1.275\mathrm{e}{-6}\;(4.60)$ & $7.716\mathrm{e}{-7}\;(6.05)$ & $3.020\mathrm{e}{-7}\;(5.34)$ & $4.043\mathrm{e}{-5}\;(4.91)$ \\
4 & 4 & 128 & $3.400\mathrm{e}{-8}\;(4.54)$ & $6.024\mathrm{e}{-8}\;(4.60)$ & $5.366\mathrm{e}{-8}\;(4.57)$ & $2.265\mathrm{e}{-8}\;(5.09)$ & $7.376\mathrm{e}{-9}\;(5.36)$ & $2.437\mathrm{e}{-6}\;(4.05)$ \\
\midrule
\multicolumn{9}{@{}l}{\textit{Odd magnetic degree}} \\
1 & 1 & 16 & $2.784\mathrm{e}{-2}\;(--)$ & $2.177\mathrm{e}{-1}\;(--)$ & $4.967\mathrm{e}{-2}\;(--)$ & $1.115\mathrm{e}{-1}\;(--)$ & $4.297\mathrm{e}{-2}\;(--)$ & $3.110\mathrm{e}{-1}\;(--)$ \\
1 & 1 & 32 & $1.232\mathrm{e}{-2}\;(1.18)$ & $9.285\mathrm{e}{-2}\;(1.23)$ & $2.043\mathrm{e}{-2}\;(1.28)$ & $5.372\mathrm{e}{-2}\;(1.05)$ & $1.214\mathrm{e}{-2}\;(1.82)$ & $2.786\mathrm{e}{-1}\;(0.16)$ \\
1 & 1 & 64 & $4.064\mathrm{e}{-3}\;(1.60)$ & $2.904\mathrm{e}{-2}\;(1.68)$ & $6.151\mathrm{e}{-3}\;(1.73)$ & $2.625\mathrm{e}{-2}\;(1.03)$ & $3.544\mathrm{e}{-3}\;(1.78)$ & $2.632\mathrm{e}{-1}\;(0.08)$ \\
1 & 1 & 128 & $1.113\mathrm{e}{-3}\;(1.87)$ & $7.781\mathrm{e}{-3}\;(1.90)$ & $1.599\mathrm{e}{-3}\;(1.94)$ & $1.306\mathrm{e}{-2}\;(1.01)$ & $9.518\mathrm{e}{-4}\;(1.90)$ & $2.597\mathrm{e}{-1}\;(0.02)$ \\
\addlinespace[1.5pt]
2 & 1 & 16 & $1.065\mathrm{e}{-2}\;(--)$ & $3.701\mathrm{e}{-2}\;(--)$ & $1.683\mathrm{e}{-2}\;(--)$ & $1.099\mathrm{e}{-1}\;(--)$ & $4.285\mathrm{e}{-2}\;(--)$ & $3.061\mathrm{e}{-1}\;(--)$ \\
2 & 1 & 32 & $6.443\mathrm{e}{-3}\;(0.73)$ & $1.054\mathrm{e}{-2}\;(1.81)$ & $1.069\mathrm{e}{-2}\;(0.65)$ & $5.068\mathrm{e}{-2}\;(1.12)$ & $1.011\mathrm{e}{-2}\;(2.08)$ & $2.640\mathrm{e}{-1}\;(0.21)$ \\
2 & 1 & 64 & $4.687\mathrm{e}{-3}\;(0.46)$ & $6.700\mathrm{e}{-3}\;(0.65)$ & $7.788\mathrm{e}{-3}\;(0.46)$ & $2.158\mathrm{e}{-2}\;(1.23)$ & $2.441\mathrm{e}{-3}\;(2.05)$ & $2.172\mathrm{e}{-1}\;(0.28)$ \\
2 & 1 & 128 & $2.895\mathrm{e}{-3}\;(0.70)$ & $3.968\mathrm{e}{-3}\;(0.76)$ & $4.800\mathrm{e}{-3}\;(0.70)$ & $7.914\mathrm{e}{-3}\;(1.45)$ & $5.250\mathrm{e}{-4}\;(2.22)$ & $1.577\mathrm{e}{-1}\;(0.46)$ \\
\addlinespace[1.5pt]
3 & 3 & 16 & $1.301\mathrm{e}{-3}\;(--)$ & $6.320\mathrm{e}{-3}\;(--)$ & $2.229\mathrm{e}{-3}\;(--)$ & $7.851\mathrm{e}{-3}\;(--)$ & $2.704\mathrm{e}{-3}\;(--)$ & $7.654\mathrm{e}{-2}\;(--)$ \\
3 & 3 & 32 & $9.096\mathrm{e}{-5}\;(3.84)$ & $4.825\mathrm{e}{-4}\;(3.71)$ & $1.588\mathrm{e}{-4}\;(3.81)$ & $4.390\mathrm{e}{-4}\;(4.16)$ & $1.246\mathrm{e}{-4}\;(4.44)$ & $7.230\mathrm{e}{-3}\;(3.40)$ \\
3 & 3 & 64 & $1.162\mathrm{e}{-5}\;(2.97)$ & $2.356\mathrm{e}{-5}\;(4.36)$ & $1.881\mathrm{e}{-5}\;(3.08)$ & $5.447\mathrm{e}{-5}\;(3.01)$ & $6.987\mathrm{e}{-6}\;(4.16)$ & $1.692\mathrm{e}{-3}\;(2.10)$ \\
3 & 3 & 128 & $1.516\mathrm{e}{-6}\;(2.94)$ & $2.220\mathrm{e}{-6}\;(3.41)$ & $2.523\mathrm{e}{-6}\;(2.90)$ & $6.705\mathrm{e}{-6}\;(3.02)$ & $3.684\mathrm{e}{-7}\;(4.25)$ & $4.075\mathrm{e}{-4}\;(2.05)$ \\
\addlinespace[1.5pt]
4 & 3 & 16 & $3.706\mathrm{e}{-4}\;(--)$ & $1.677\mathrm{e}{-3}\;(--)$ & $6.480\mathrm{e}{-4}\;(--)$ & $6.880\mathrm{e}{-3}\;(--)$ & $2.277\mathrm{e}{-3}\;(--)$ & $6.659\mathrm{e}{-2}\;(--)$ \\
4 & 3 & 32 & $1.087\mathrm{e}{-4}\;(1.77)$ & $1.550\mathrm{e}{-4}\;(3.44)$ & $1.798\mathrm{e}{-4}\;(1.85)$ & $3.916\mathrm{e}{-4}\;(4.13)$ & $8.726\mathrm{e}{-5}\;(4.71)$ & $6.452\mathrm{e}{-3}\;(3.37)$ \\
4 & 3 & 64 & $2.053\mathrm{e}{-5}\;(2.41)$ & $2.781\mathrm{e}{-5}\;(2.48)$ & $3.395\mathrm{e}{-5}\;(2.40)$ & $4.693\mathrm{e}{-5}\;(3.06)$ & $4.327\mathrm{e}{-6}\;(4.33)$ & $1.451\mathrm{e}{-3}\;(2.15)$ \\
4 & 3 & 128 & $2.614\mathrm{e}{-6}\;(2.97)$ & $3.492\mathrm{e}{-6}\;(2.99)$ & $4.320\mathrm{e}{-6}\;(2.97)$ & $4.680\mathrm{e}{-6}\;(3.33)$ & $1.864\mathrm{e}{-7}\;(4.54)$ & $2.839\mathrm{e}{-4}\;(2.35)$ \\
\bottomrule
\end{tabular}
}
\end{table}

Table~\ref{tab:balsara-vortex-spatial-parity} confirms that the parity effect
persists when both split subflows are active. At \(p=1\), the even choice
\(m=0\) reduces the finest-grid current error by a factor of 13.5 relative to
\(m=1\) and gives a finest-pair current rate of 1.92, compared with 0.02 for
the odd choice; thus, the even choice is nearly second-order accurate in the
current, whereas the odd choice shows no current convergence. At \(p=2\),
the even choice \(m=2\) gives rates of about
2.5 in the fluid variables, above third order in \(B\) and \(E_z\), and
second order in \(J_z\). The odd choice \(m=1\), by contrast, is sub-first
order in density, velocity, and pressure on the finest refinement and reaches
only 0.46 in \(J_z\).

For \(p=3\), both magnetic degrees recover approximately third-order accuracy
in the primary variables, although \(m=3\) gives smaller absolute errors. The
ordering reverses at \(p=4\): the even choice \(m=4\) gives rates
near 4.5 in the fluid variables, above five in \(B\) and \(E_z\), and 4.05 in
\(J_z\), whereas \(m=3\) is limited to approximately third order in the fluid
and magnetic variables and 2.35 in the current. Overall, the even-degree
family gives the more reliable high-order behavior.

\subsubsection{Temporal convergence}
\label{sec:balsara-vortex-temporal-convergence}

We next isolate the temporal error by fixing \((p,m)=(2,2)\) and a
\(64\times64\) mesh. Six solutions are evolved to \(t_f=1\) with
\(\Delta t_\ell=10^{-2}2^{-\ell}\), \(\ell=0,\ldots,5\). Since the direct
errors reach the fixed spatial-error floor, we measure temporal
self-convergence using the same-mesh Cauchy differences
\[
 d_q^\ell
 =\left\lVert q_h^{\Delta t_\ell}
       -q_h^{\Delta t_{\ell+1}}\right\rVert_{L^2(\Omega_e)},
 \qquad
 r_q^\ell=\log_2\!\left(\frac{d_q^{\ell-1}}{d_q^\ell}\right),
 \quad \ell\geq1.
\]
All stabilization mechanisms remain disabled.

\begin{table}[H]
\centering
\caption{Temporal self-convergence of the advected MHD vortex for
$(p,m)=(2,2)$ on a fixed $64\times64$ mesh at $t_f=1$. Here
$d_q(\Delta t)=\lVert q_h^{\Delta t}-q_h^{\Delta t/2}\rVert_{L^2(\Omega_e)}$.
Each entry gives the Cauchy difference followed by the temporal rate in
parentheses.}
\label{tab:balsara-vortex-temporal}
\scriptsize
\setlength{\tabcolsep}{3.0pt}
\resizebox{\linewidth}{!}{%
\begin{tabular}{@{}rcccccc@{}}
\toprule
$\Delta t$ & $d_\rho$ & $d_u$ & $d_p$ & $d_B$ & $d_{E_z}$ & $d_{J_z}$ \\
\midrule
$1.000\mathrm{e}{-2}$ & $2.882\mathrm{e}{-6}\;(--)$ & $3.526\mathrm{e}{-6}\;(--)$ & $4.776\mathrm{e}{-6}\;(--)$ & $1.277\mathrm{e}{-6}\;(--)$ & $1.690\mathrm{e}{-6}\;(--)$ & $6.982\mathrm{e}{-6}\;(--)$ \\
$5.000\mathrm{e}{-3}$ & $7.189\mathrm{e}{-7}\;(2.00)$ & $8.682\mathrm{e}{-7}\;(2.02)$ & $1.190\mathrm{e}{-6}\;(2.00)$ & $3.128\mathrm{e}{-7}\;(2.03)$ & $4.159\mathrm{e}{-7}\;(2.02)$ & $1.588\mathrm{e}{-6}\;(2.14)$ \\
$2.500\mathrm{e}{-3}$ & $1.799\mathrm{e}{-7}\;(2.00)$ & $2.161\mathrm{e}{-7}\;(2.01)$ & $2.977\mathrm{e}{-7}\;(2.00)$ & $7.765\mathrm{e}{-8}\;(2.01)$ & $1.034\mathrm{e}{-7}\;(2.01)$ & $3.847\mathrm{e}{-7}\;(2.05)$ \\
$1.250\mathrm{e}{-3}$ & $4.500\mathrm{e}{-8}\;(2.00)$ & $5.394\mathrm{e}{-8}\;(2.00)$ & $7.446\mathrm{e}{-8}\;(2.00)$ & $1.936\mathrm{e}{-8}\;(2.00)$ & $2.580\mathrm{e}{-8}\;(2.00)$ & $9.513\mathrm{e}{-8}\;(2.02)$ \\
$6.250\mathrm{e}{-4}$ & $1.126\mathrm{e}{-8}\;(2.00)$ & $1.348\mathrm{e}{-8}\;(2.00)$ & $1.862\mathrm{e}{-8}\;(2.00)$ & $4.833\mathrm{e}{-9}\;(2.00)$ & $6.443\mathrm{e}{-9}\;(2.00)$ & $2.368\mathrm{e}{-8}\;(2.01)$ \\
\bottomrule
\end{tabular}
}
\end{table}

Table~\ref{tab:balsara-vortex-temporal} shows uniform second-order convergence
in every reported field, confirming the expected temporal accuracy of the
fully coupled Strang composition.

\subsubsection{Spatial accuracy with stabilization}
\label{sec:balsara-vortex-stabilized-accuracy}

In all subsequent tests, we use the equal-order choice \((p,m)=(2,2)\). We
first verify on the smooth vortex that OE and artificial resistivity preserve
formal spatial accuracy. The refinement study uses \(N=16,32,64,128\),
\(\Delta t=10^{-4}\), and \(t_f=1\). We compare otherwise identical versions
of the method without stabilization, with OE and positivity limiting, and with
OE, positivity limiting, and A.R.

\begin{table}[H]
\centering
\caption{Spatial convergence of the advected MHD vortex for
$(p,m)=(2,2)$ without stabilization, with OE, and with OE plus artificial
resistivity (A.R.). The stabilized configurations use $s_H=s_M=0.02$ and
the positivity limiter; the last also uses $C_\eta=0.2$ and $C_s=5$. Each
configuration uses the same hydrodynamic--magnetic--hydrodynamic Strang
composition and spatial discretization. Each entry gives the $L^2$ error
followed by the successive-refinement rate in parentheses.}
\label{tab:balsara-vortex-stabilized-accuracy}
\scriptsize
\setlength{\tabcolsep}{1.8pt}
\resizebox{\linewidth}{!}{%
\begin{tabular}{@{}rcccccc@{}}
\toprule
$N$ & $e_\rho$ & $e_u$ & $e_p$ & $e_B$ & $e_{E_z}$ & $e_{J_z}$ \\
\midrule
\multicolumn{7}{@{}l}{\textit{No stabilization} (OE and A.R. off)} \\
16  & $9.863\mathrm{e}{-3}\;(--)$ & $3.668\mathrm{e}{-2}\;(--)$ & $1.345\mathrm{e}{-2}\;(--)$ & $8.791\mathrm{e}{-2}\;(--)$ & $1.607\mathrm{e}{-2}\;(--)$ & $4.256\mathrm{e}{-1}\;(--)$ \\
32  & $1.522\mathrm{e}{-3}\;(2.70)$ & $4.669\mathrm{e}{-3}\;(2.97)$ & $2.587\mathrm{e}{-3}\;(2.38)$ & $7.262\mathrm{e}{-3}\;(3.60)$ & $2.071\mathrm{e}{-3}\;(2.96)$ & $7.411\mathrm{e}{-2}\;(2.52)$ \\
64  & $2.829\mathrm{e}{-4}\;(2.43)$ & $8.723\mathrm{e}{-4}\;(2.42)$ & $4.998\mathrm{e}{-4}\;(2.37)$ & $4.224\mathrm{e}{-4}\;(4.10)$ & $1.719\mathrm{e}{-4}\;(3.59)$ & $8.737\mathrm{e}{-3}\;(3.08)$ \\
128 & $5.596\mathrm{e}{-5}\;(2.34)$ & $1.442\mathrm{e}{-4}\;(2.60)$ & $1.004\mathrm{e}{-4}\;(2.32)$ & $5.054\mathrm{e}{-5}\;(3.06)$ & $1.852\mathrm{e}{-5}\;(3.21)$ & $2.152\mathrm{e}{-3}\;(2.02)$ \\
\addlinespace[1.5pt]
\multicolumn{7}{@{}l}{\textit{OE without A.R.} ($C_\eta=0$)} \\
16  & $1.107\mathrm{e}{-2}\;(--)$ & $1.679\mathrm{e}{-1}\;(--)$ & $2.105\mathrm{e}{-2}\;(--)$ & $8.840\mathrm{e}{-2}\;(--)$ & $1.627\mathrm{e}{-2}\;(--)$ & $4.283\mathrm{e}{-1}\;(--)$ \\
32  & $4.056\mathrm{e}{-3}\;(1.45)$ & $4.371\mathrm{e}{-2}\;(1.94)$ & $7.408\mathrm{e}{-3}\;(1.51)$ & $7.590\mathrm{e}{-3}\;(3.54)$ & $2.289\mathrm{e}{-3}\;(2.83)$ & $7.695\mathrm{e}{-2}\;(2.48)$ \\
64  & $8.301\mathrm{e}{-4}\;(2.29)$ & $7.828\mathrm{e}{-3}\;(2.48)$ & $1.517\mathrm{e}{-3}\;(2.29)$ & $5.855\mathrm{e}{-4}\;(3.70)$ & $3.607\mathrm{e}{-4}\;(2.67)$ & $9.626\mathrm{e}{-3}\;(3.00)$ \\
128 & $9.504\mathrm{e}{-5}\;(3.13)$ & $6.243\mathrm{e}{-4}\;(3.65)$ & $1.685\mathrm{e}{-4}\;(3.17)$ & $6.440\mathrm{e}{-5}\;(3.18)$ & $3.720\mathrm{e}{-5}\;(3.28)$ & $2.271\mathrm{e}{-3}\;(2.08)$ \\
\addlinespace[1.5pt]
\multicolumn{7}{@{}l}{\textit{OE with A.R.} ($C_\eta=0.2$, $C_s=5$)} \\
16  & $8.579\mathrm{e}{-3}\;(--)$ & $1.666\mathrm{e}{-1}\;(--)$ & $2.257\mathrm{e}{-2}\;(--)$ & $7.196\mathrm{e}{-2}\;(--)$ & $5.343\mathrm{e}{-2}\;(--)$ & $2.496\mathrm{e}{-1}\;(--)$ \\
32  & $3.522\mathrm{e}{-3}\;(1.28)$ & $4.366\mathrm{e}{-2}\;(1.93)$ & $7.420\mathrm{e}{-3}\;(1.60)$ & $9.542\mathrm{e}{-3}\;(2.91)$ & $8.639\mathrm{e}{-3}\;(2.63)$ & $6.628\mathrm{e}{-2}\;(1.91)$ \\
64  & $7.957\mathrm{e}{-4}\;(2.15)$ & $7.828\mathrm{e}{-3}\;(2.48)$ & $1.497\mathrm{e}{-3}\;(2.31)$ & $7.238\mathrm{e}{-4}\;(3.72)$ & $5.544\mathrm{e}{-4}\;(3.96)$ & $9.625\mathrm{e}{-3}\;(2.78)$ \\
128 & $9.353\mathrm{e}{-5}\;(3.09)$ & $6.244\mathrm{e}{-4}\;(3.65)$ & $1.674\mathrm{e}{-4}\;(3.16)$ & $6.891\mathrm{e}{-5}\;(3.39)$ & $4.572\mathrm{e}{-5}\;(3.60)$ & $2.269\mathrm{e}{-3}\;(2.08)$ \\
\bottomrule
\end{tabular}
}
\end{table}

Table~\ref{tab:balsara-vortex-stabilized-accuracy} shows comparable
asymptotic behavior for all three configurations. On the finest refinement,
the unstabilized primary-variable rates are 2.32--3.21, with a rate of 2.02
for the differentiated current \(J_z\). With OE alone, these rates are
3.13--3.65 and 2.08, respectively; with OE and A.R., they are 3.09--3.65 and
2.08. Thus, neither stabilization configuration reduces the observed order,
and the two stabilized runs have comparable finest-grid errors. Moreover,
the final-step maximum A.R. coefficient decreases monotonically from
\(6.46\times10^{-2}\) at \(N=16\) to \(3.48\times10^{-5}\) at \(N=128\),
with rates 3.99 and 4.13 over the final two refinements. The rapid decay of
this coefficient shows that the magnetic sensor becomes asymptotically
inactive as the smooth solution is resolved. Together, these results confirm
that OE and A.R. retain the observed high-order accuracy of the underlying
discretization on smooth solutions.

\subsection{Advection of a magnetic-field loop}
\label{sec:fieldloop}

The weak-field loop-advection problem introduced by Gardiner and
Stone~\cite{gardiner2005unsplit} is a standard test of multidimensional
magnetic transport and divergence control
\cite{stone2008athena,guillet2019highorder,seo2023howmhd}. The magnetic field
is effectively advected as a passive quantity, while the discontinuities in
its derivatives at the center and perimeter make the current density a
particularly sensitive diagnostic of numerical noise.

We use the periodic domain
\(\Omega=[-1,1]\times[-1/2,1/2]\), with
\(\rho=p=1\), \(\gamma=5/3\), and
\(\boldsymbol u=(2,1,0)^T\). The loop is initialized through
\[
 \boldsymbol B=\nabla\times(A_z\boldsymbol e_z),
 \qquad
 A_z(r)=A_0\max(R-r,0),
 \qquad
 A_0=10^{-3},\quad R=0.3,
\]
where \(r=(x^2+y^2)^{1/2}\). Since the fluid traverses both periodic lengths
in one unit of time, the exact loop returns to its initial position at
\(t_f=1\). The calculations in Fig.~\ref{fig:fieldloop-ar} use a
\(256\times128\) mesh and CFL number 0.8.

\begin{figure*}[t]
 \centering
 \includegraphics[width=\textwidth]{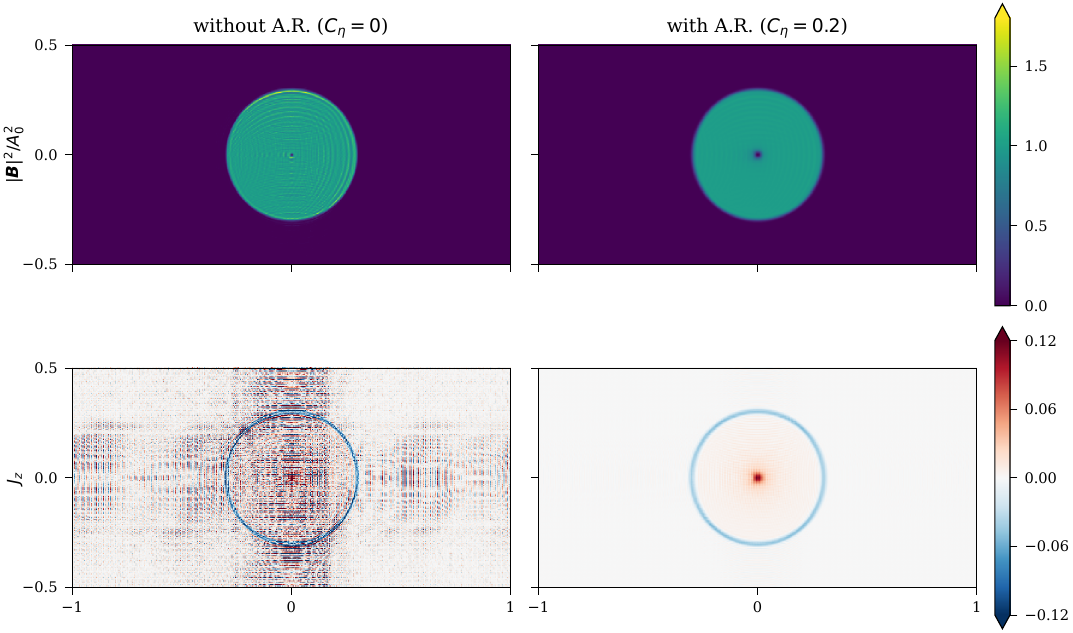}
 \caption{Advected magnetic-field loop at \(t=1\) after one traversal of the
 periodic domain. The top row shows the normalized magnetic energy, and the
 bottom row shows the current density. The left and right columns respectively
 omit and include artificial resistivity. Each row uses a common color scale.
 To resolve the structure of the current ring, the \(J_z\) scale is saturated at
 \(\lvert J_z\rvert=0.12\); the actual extrema are
 \([-0.349,0.482]\) without A.R. and \([-0.0537,0.119]\) with A.R.}
 \label{fig:fieldloop-ar}
\end{figure*}

Without A.R., the magnetic field remains sharply confined but develops a
pronounced overshoot at the loop perimeter, and the differentiated current
reveals grid-aligned oscillations throughout the domain. Artificial
resistivity suppresses these oscillations and recovers a clean central
current and circular return-current ring. Quantitatively, the maximum field
amplitude decreases from \(1.412A_0\) without A.R. to \(1.018A_0\) with
A.R. The compatible update retains its divergence control: the final values
of \(\|\nabla\cdot\boldsymbol B_h\|_{L^2}\) are
\(2.44\times10^{-15}\) and \(2.42\times10^{-15}\), respectively. This
improvement in nonoscillatory behavior comes with additional magnetic
dissipation. The retained magnetic energy decreases from \(99.87\%\) without
A.R. to \(96.45\%\) with A.R.
To measure the error in the transported magnetic profile, we use
\[
 e_B^{\rm tr}(t)
 =\frac{\left\lVert\boldsymbol B_h(\cdot,t)
      -\boldsymbol B(\cdot-\boldsymbol u t,0)\right\rVert_{L^2(\Omega)}}
 {\left\lVert\boldsymbol B(\cdot-\boldsymbol u t,0)\right\rVert_{L^2(\Omega)}},
\]
where the translation is interpreted periodically. At \(t=1\), this target
is the initial analytic loop. The errors are 0.11827 without A.R. and 0.11839
with A.R. Their near equality indicates that A.R. replaces grid-scale
oscillatory error with smoother dissipative error while leaving the overall
relative \(L^2\)-error magnitude essentially unchanged.

\subsection{Orszag--Tang vortex}
\label{sec:orszag-tang}

The Orszag--Tang vortex~\cite{orszag1979small} evolves
smooth initial data into interacting shocks and current sheets. On the
periodic domain \([0,2\pi]^2\), the initial conditions are
\begin{align*}
 \rho&=\gamma^2, & p&=\gamma, & \gamma&=5/3,\\
 \boldsymbol u&=(-\sin y,\sin x)^T, &
 \boldsymbol B&=(-\sin y,\sin(2x))^T.&&
\end{align*}
We apply Algorithm~\ref{alg:strang-step} with polynomial degrees \(p=m=2\) on
a \(256\times256\) mesh, with and without A.R., and evolve both solutions to
\(t=3\). Figure~\ref{fig:ot-ar} shows equally spaced contours of the density,
magnetic pressure \(p_m=\lvert\boldsymbol B\rvert^2/2\), and acoustic Mach
number \(M=\lvert\boldsymbol u\rvert/\sqrt{\gamma p/\rho}\). For each quantity,
identical contour levels are used in the two calculations to enable a direct
comparison.

\begin{figure*}[t]
 \centering
 \includegraphics[width=\textwidth]{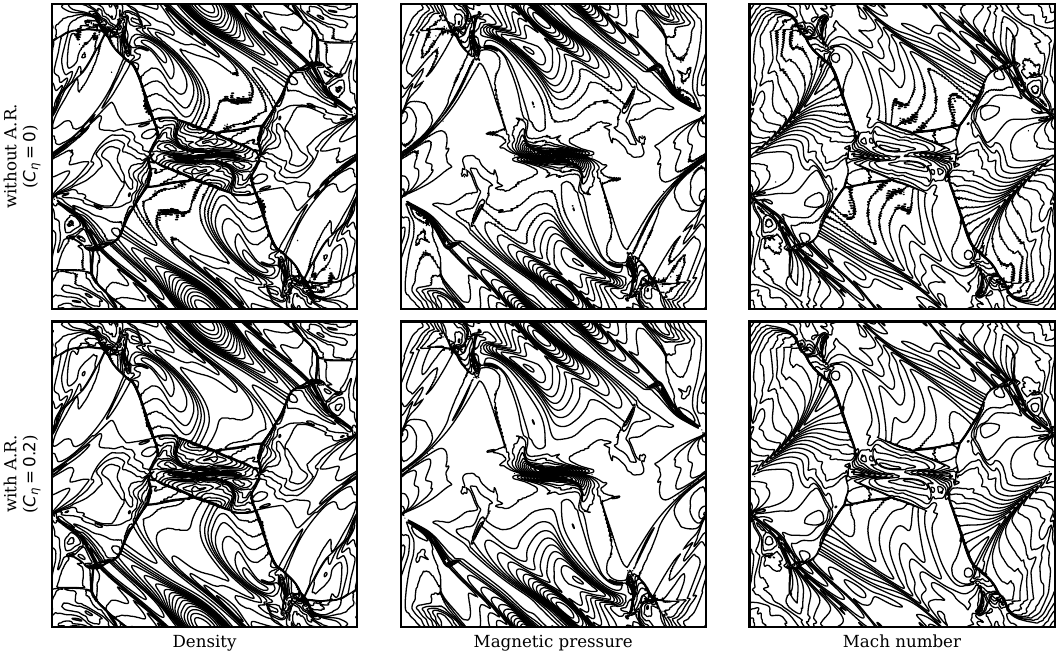}
 \caption{Orszag--Tang vortex at \(t=3\) on a \(256\times256\) mesh.
 Twenty contour lines are shown for density \(\rho\) (left), magnetic
 pressure \(p_m\) (middle), and acoustic Mach number \(M\) (right). The top
 and bottom rows use \(C_\eta=0\) and \(C_\eta=0.2\), respectively. In both
 rows, the 20 common contour levels span \([1.066,6.163]\) for \(\rho\),
 \([0.001,5.060]\) for \(p_m\), and \([0.001,2.614]\) for \(M\).}
 \label{fig:ot-ar}
\end{figure*}

The solutions in Fig.~\ref{fig:ot-ar} have the same shock topology and nearly
identical density maxima: 6.1604 without A.R. and 6.1594 with A.R. Their
magnetic-pressure and Mach-number contours retain the dominant sheets,
compression zones, and interacting wave fronts. In the
nonresistive calculation, small grid-scale contour corrugations appear near
several shock and current-sheet intersections; A.R. suppresses these features
without broadening the resolved large-scale structures. The final value of
\(\|\nabla\cdot\boldsymbol B_h\|_{L^2}\) is \(1.14\times10^{-11}\) in both
calculations.

\subsection{MHD rotor}
\label{sec:mhd-rotor}

The two-dimensional MHD rotor benchmark was introduced by Balsara and
Spicer~\cite{balsara1999staggered}. The test probes strong torsional Alfv\'en
waves launched by a dense, rapidly rotating core into a quiescent magnetized
fluid. On \(\Omega=[0,1]^2\), we impose mirror boundary conditions and center
the rotor at \((1/2,1/2)\). It has density 10 for \(r<0.1\) and rotates rigidly
with unit tangential speed at \(r=0.1\). Over \(0.1\leq r<0.115\), its density
and tangential speed taper linearly to the ambient values \(\rho=1\) and
\(\boldsymbol u=\boldsymbol0\). The initial pressure is \(p=0.5\), the
adiabatic index is \(\gamma=5/3\), and the uniform magnetic field is
\(\boldsymbol B=(2.5/\sqrt{4\pi},0)^T\). We apply
Algorithm~\ref{alg:strang-step} with \(p=m=2\) on a \(256\times256\) mesh,
with and without A.R., and evolve both solutions to \(t=0.295\).

\begin{figure*}[t]
 \centering
 \includegraphics[width=\textwidth]{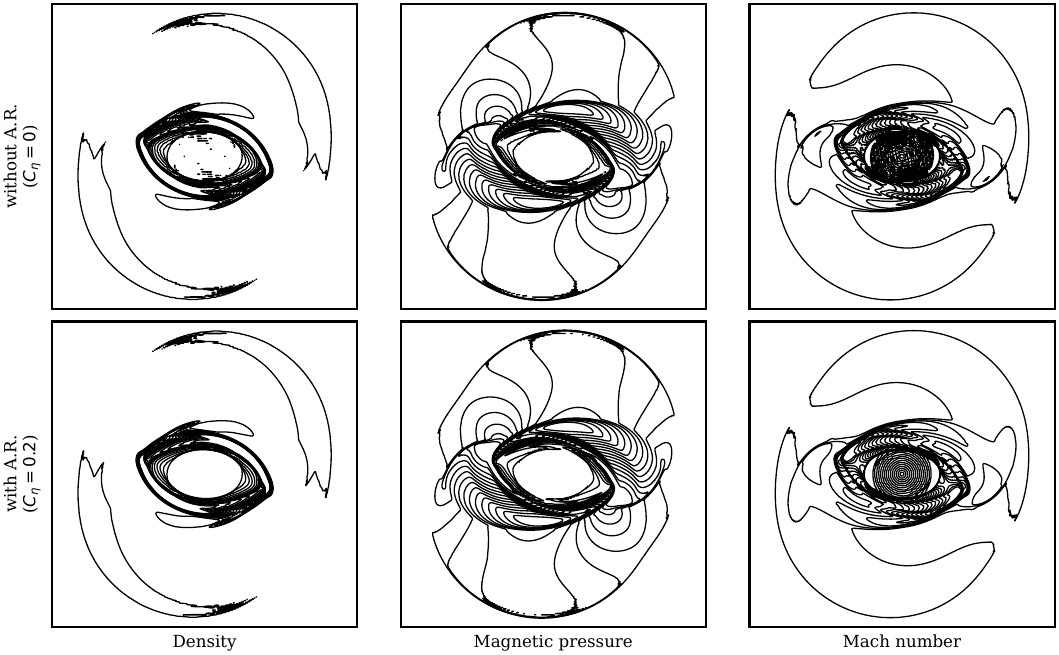}
 \caption{MHD rotor at \(t=0.295\) on a \(256\times256\) mesh. Fifteen
 contour lines are shown for density \(\rho\) (left), magnetic pressure
 \(p_m\) (middle), and acoustic Mach number \(M\) (right). The top and bottom
 rows use \(C_\eta=0\) and \(C_\eta=0.2\), respectively. In both rows, the 15
 common contour levels span \([0.730,7.330]\) for \(\rho\),
 \([0.059,0.655]\) for \(p_m\), and
 \([0.147,2.270]\) for \(M\).}
 \label{fig:rotor-ar}
\end{figure*}

Both calculations in Fig.~\ref{fig:rotor-ar} resolve the expected magnetic
braking and deformation of the dense core. The large-scale density,
magnetic-pressure, and Mach-number contours agree closely. Without A.R., a
distinct grid-scale pattern appears inside the rotor, most clearly in the
density and Mach-number contours. A.R. removes this contamination while
preserving the compressed magnetic layers and the overall rotor geometry.
The actual magnetic-pressure maxima, slightly above the displayed contour
range, remain nearly unchanged: 0.7051 without A.R. and 0.7038 with A.R. The
final-step maximum A.R. coefficient is
\(8.63\times10^{-5}\). The final divergence diagnostics also agree:
\(\|\nabla\cdot\boldsymbol B_h\|_{L^2}=4.60\times10^{-12}\) for both
calculations. Thus, the added dissipation is localized by the magnetic sensor
and does not compromise the compatible divergence control.

\subsection{Oblique MHD blast wave}
\label{sec:mhd-blast}

The magnetized blast wave is a stringent multidimensional test because a
strong pressure front propagates both along and across the magnetic field
\cite{balsara1999staggered,guillet2019highorder}. On the periodic domain
\(\Omega=[0,1]^2\), we initialize \(\rho=1\),
\(\boldsymbol u=\boldsymbol0\), and
\(\boldsymbol B=(\cos(\pi/4),\sin(\pi/4))^T\). The pressure is 10 in the
disk of radius 0.1 centered at \((1/2,1/2)\) and 0.1 outside it. We take
\(\gamma=5/3\), use a \(256\times256\) mesh, and evolve both calculations to
\(t=0.2\).

\begin{figure*}[t]
 \centering
 \includegraphics[width=\textwidth]{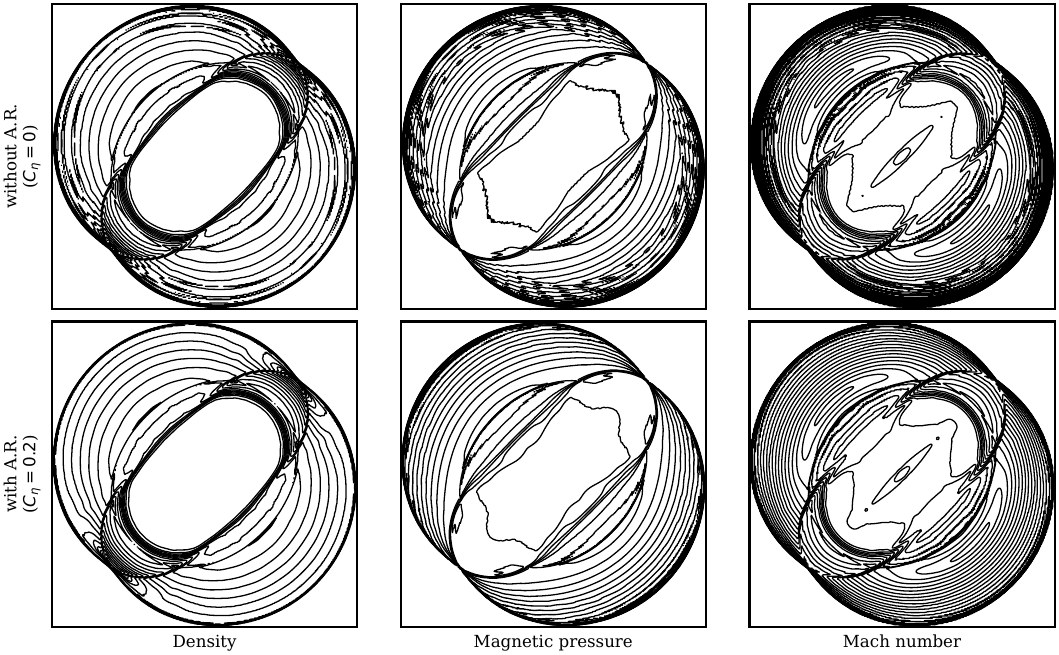}
 \caption{Oblique MHD blast wave at \(t=0.2\) on a \(256\times256\) mesh.
 Twenty contour lines are shown for density \(\rho\) (left), magnetic
 pressure \(p_m\) (middle), and acoustic Mach number \(M\) (right). The top
 and bottom rows use \(C_\eta=0\) and \(C_\eta=0.2\), respectively. Each
 column uses the same contour levels in both rows: \([0.200,2.800]\) for
 \(\rho\),
 \([0.100,1.800]\) for \(p_m\), and \([0.050,1.500]\) for \(M\).}
 \label{fig:blast-ar}
\end{figure*}

The density, magnetic-pressure, and Mach-number contours in
Fig.~\ref{fig:blast-ar} retain the same anisotropic shock geometry with and
without A.R. The nonresistive calculation develops fine contour corrugations
along the curved outer front and weaker grid-scale structure behind it. A.R.
strongly attenuates these oscillations while retaining the leading wave fronts
and the field-aligned structure of the blast. The final-step maximum A.R.
coefficient is
\(3.18\times10^{-4}\); the associated regularization reduces the
magnetic-pressure maximum from 2.25 to 1.77. The final divergence norm is
\(\|\nabla\cdot\boldsymbol B_h\|_{L^2}=1.04\times10^{-11}\) in both runs,
demonstrating that the curl-form regularization leaves magnetic compatibility
intact.

\subsection{Oblique-field Kelvin--Helmholtz instability}
\label{sec:mhd-kh}

Finally, we consider Case~3 of Jones et al.~\cite{jones1997mhd}, a 2.5D
Kelvin--Helmholtz instability with a weak oblique magnetic field. The square
domain \([0,L]^2\), \(L=2.51\), is periodic in \(x\) and reflecting in \(y\).
The unperturbed shear layer is
\begin{align*}
 \rho_0&=1, & p_0&=0.6,\\
 u_{x,0}&=-0.5\tanh\!\left(\frac{y-L/2}{0.1004}\right), &
 \boldsymbol B_0&=0.2(\cos45^\circ,0,\sin45^\circ)^T,
\end{align*}
with \(\gamma=5/3\). We perturb the primitive state
\(\boldsymbol q=(\rho,\boldsymbol u,p,\boldsymbol B)\) by its fastest-growing
fundamental streamwise normal mode,
\(\delta\boldsymbol q(x,y,t)=\operatorname{Re}\{\widehat{\boldsymbol q}(y)
e^{\mathrm{i}kx+\Gamma t}\}\), with \(k=2\pi/L\). The eigenfunction is
computed from the one-dimensional linearized compressible ideal-MHD problem
using 160 Chebyshev--Lobatto intervals. We impose reflecting-wall parity
(\(\widehat u_y,\widehat B_y\) odd and the remaining primitive components
even) and enforce
\(\mathrm{i}k\widehat B_x+\partial_y\widehat B_y=0\). Normalizing the mode by
\(\lvert\widehat{\delta p}_{\mathrm{tot}}(L)\rvert
=0.01p_{\mathrm{tot},0}\), where
\(p_{\mathrm{tot},0}=p_0+\lvert\boldsymbol B_0\rvert^2/2\), gives
\(\Gamma=0.58370\) and \(t_g=\Gamma^{-1}\simeq1.71\). We report the
normalized time \(\tau=t/t_g\). We apply Algorithm~\ref{alg:strang-step} with
polynomial degree \(p=m=2\) on a \(256\times256\) mesh, with and without
A.R., and evolve both calculations to \(\tau=10\).

\begin{figure*}[t]
 \centering
 \includegraphics[width=\textwidth]{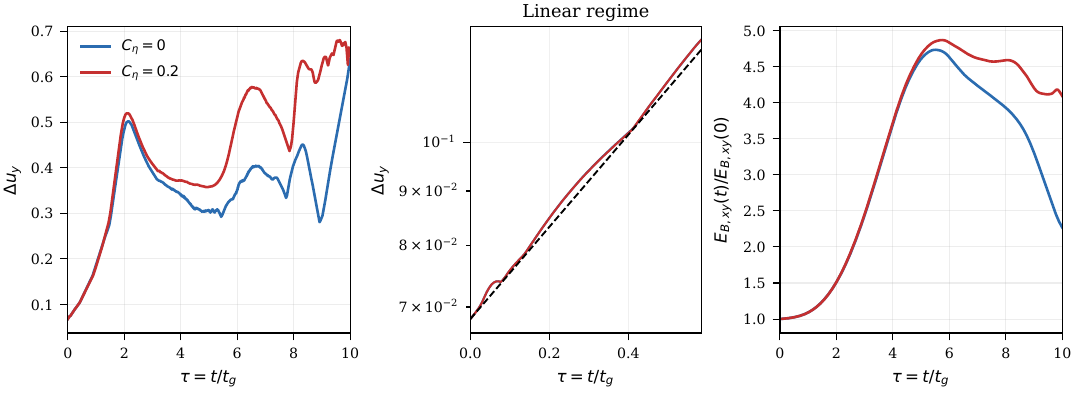}
 \caption{Kelvin--Helmholtz evolution with and without A.R. Left: transverse
 velocity amplitude
 \(\Delta u_y=(\max_\Omega u_y-\min_\Omega u_y)/2\) over the full simulation.
 Center: early-time \(\Delta u_y\) on a logarithmic scale; the dashed line is
 the linear prediction \(\Delta u_y(0)e^{\Gamma t}\). Right: poloidal magnetic
 energy normalized by its initial value. The two histories agree through the
 linear regime and separate only after nonlinear roll-up and current-sheet
 formation.}
 \label{fig:kh-diagnostics}
\end{figure*}

The fitted growth rates over \(0\leq t\leq1\) are 0.59096 for
\(C_\eta=0\) and 0.59102 for \(C_\eta=0.2\), compared with the computed
linear eigenvalue 0.58370. Thus, both are within 1.3\% of linear theory, and
A.R. does not measurably alter the seeded instability during its smooth
linear phase. In the nonlinear regime, magnetic stretching increases the
poloidal magnetic energy to peak factors of 4.74 and 4.87 without and with
A.R., respectively. After the peak, the nonresistive calculation loses
poloidal magnetic energy more rapidly, ending at factors of 2.27 and 4.09,
respectively, whereas the final transverse-velocity amplitudes remain close
(0.639 and 0.645). Although A.R. is locally dissipative, the larger retained
magnetic energy results from its modification of reconnection and subsequent
magnetic stretching.

\begin{figure*}[t]
 \centering
 \includegraphics[width=\textwidth]{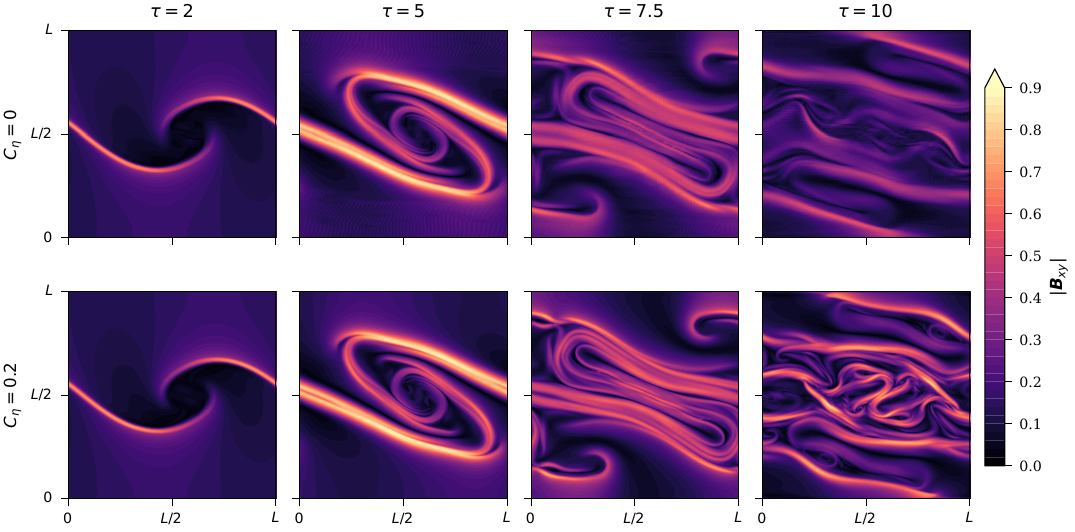}
 \caption{Magnitude of the poloidal magnetic field for the
 Kelvin--Helmholtz problem at \(\tau=2,5,7.5,10\). The top row omits A.R.; the
 bottom row uses \(C_\eta=0.2\). One common color scale is used in all panels.
 The solutions agree through the primary roll-up. At later times, A.R.
 reduces the faint grid-aligned texture visible in the nonresistive result
 and changes the secondary-sheet and filament topology.}
 \label{fig:kh-snapshots}
\end{figure*}

Figure~\ref{fig:kh-snapshots} shows
\(\lvert\boldsymbol B_{xy}\rvert\) at four times, resolving the initial
roll-up, magnetic-field stretching, and subsequent formation of secondary
sheets and filaments. The two calculations remain visually close through
\(\tau=5\). At \(\tau=7.5\), the nonresistive result develops faint, nearly
grid-aligned texture away from the principal magnetic layers, which is
reduced by A.R. By \(\tau=10\), the filament topologies differ substantially
because reconnection has redirected the nonlinear evolution; this separation
should therefore be interpreted as regularization sensitivity. The
final-step maximum A.R. coefficient is \(2.94\times10^{-4}\). In both cases,
the relative total-energy change is below \(1.6\times10^{-4}\), and the final
value of \(\|\nabla\cdot\boldsymbol B_h\|_{L^2}\) is approximately
\(1.1\times10^{-11}\).

\section{Conclusions}
\label{sec:conclusion}

We have developed a high-order, fully explicit operator-splitting method for
2.5D ideal MHD on conforming Cartesian rectangular meshes. The hydrodynamic
subflow is advanced by an entropy-stable DGSEM with OEDG stabilization and a
positivity limiter, whereas the magnetic--velocity subflow uses compatible
finite elements for the magnetic field, electric field, and current density.
Diagonal collocated or mass-lumped products make the auxiliary
reconstructions pointwise, and the magnetic mass matrix is applied but never
inverted. The two maps are coupled by a hydrodynamic--magnetic--hydrodynamic
Strang composition.

The analysis distinguishes semidiscrete identities from properties of the
stabilized, fully discrete algorithm. The Euler semidiscretization is
conservative and satisfies an entropy inequality; under the stated
time-step condition, its forward-Euler cell averages remain admissible and
the scaling limiter restores nodal admissibility. The compatible magnetic
semidiscretization preserves the global \(H(\mathrm{div})\) divergence-free
subspace. Under periodic boundary conditions, its ideal coupling exactly
exchanges kinetic and magnetic energy, while artificial resistivity transfers
magnetic energy to internal energy through Ohmic heating. The density-weighted
velocity filter preserves element momentum and returns its kinetic-energy
loss as a nonnegative nodal internal-energy increment. These transfers imply
positive internal energy and a discrete entropy inequality for the stabilized
magnetic stage. Consequently, accepted full steps retain global mass,
nodal admissibility, and compatible magnetic divergence. The complete method
does not exactly conserve total energy or momentum at the fully discrete
level: explicit Runge--Kutta integration does not preserve the quadratic
magnetic-stage energy in general, and the discrete Lorentz-force update is not
globally momentum conservative.

The numerical results support both the accuracy and stabilization design. The
reduced Alfv\'en wave and fully coupled advected vortex show a systematic
dependence on the magnetic-degree parity, with the even-\(m\) family giving the
more reliable high-order behavior. Most notably, the current converges for
\((p,m)=(1,0)\) but stagnates for the equal-degree pair \((1,1)\). A same-mesh
Cauchy study verifies second-order temporal convergence of the Strang
algorithm, and the stabilized vortex study shows that hydrodynamic and
magnetic OE together with artificial resistivity retain the observed spatial
orders. In the nonsmooth experiments, artificial resistivity suppresses
grid-aligned current oscillations in the advected field loop and small-scale
contour corrugations in the Orszag--Tang, rotor, and blast-wave problems without
altering their resolved large-scale structures. In the Kelvin--Helmholtz test,
it leaves the linear growth rate unchanged and removes grid-scale magnetic
banding after nonlinear roll-up. Across these calculations, the reported
magnetic-divergence diagnostics remain between approximately
\(10^{-15}\) and \(10^{-11}\).

Several directions remain for future work. A fully discrete entropy analysis
that includes Euler-stage stabilization would complement the present
semidiscrete estimate. Although effective in the reported tests, the A.R.
coefficient is empirically calibrated; alternative sensor and coefficient
designs merit investigation. The present
tensor-product mass lumping and directional OE formulation target fixed,
coordinate-aligned Cartesian meshes, and extensions to curvilinear or
nonconforming meshes will require compatible geometric constructions and
corresponding stability analysis. Further directions include extension to
full three-dimensional MHD and energy-preserving time integration.

\par\medskip
\noindent\textbf{Acknowledgments.}
This work was partially supported by the National Science Foundation
under Grant DMS-2410741.

\par\medskip
\noindent\textbf{Declaration of competing interest.}
The authors declare that they have no known competing financial interests or
personal relationships that could have appeared to influence the work reported
in this paper.

\par\medskip
\noindent\textbf{Data availability.}
Data will be made available on request.

\par\medskip
\noindent\textbf{Declaration of generative AI and AI-assisted technologies in
the manuscript preparation process.}
During the preparation of this work, the authors used ChatGPT (OpenAI) and
Codex (OpenAI) to improve the language and readability of the manuscript and
to assist with software development, respectively.  After using these tools,
the authors reviewed and edited the content as needed and take full
responsibility for the content of the published article.

\bibliographystyle{spmpsci}
\bibliography{references}

\begin{thebibliography}{10}
\providecommand{\url}[1]{{#1}}
\providecommand{\urlprefix}{URL }
\expandafter\ifx\csname urlstyle\endcsname\relax
  \providecommand{\doi}[1]{DOI~\discretionary{}{}{}#1}\else
  \providecommand{\doi}{DOI~\discretionary{}{}{}\begingroup
  \urlstyle{rm}\Url}\fi

\bibitem{anderson2021mfem}
Anderson, R., Andrej, J., Barker, A., Bramwell, J., Camier, J.S., Cerveny, J.,
  Dobrev, V., Dudouit, Y., Fisher, A., Kolev, T., Pazner, W., Stowell, M.,
  Tomov, V., Akkerman, I., Dahm, J., Medina, D., Zampini, S.: {MFEM}: A modular
  finite element methods library.
\newblock Computers \& Mathematics with Applications \textbf{81}, 42--74
  (2021).
\newblock \doi{10.1016/j.camwa.2020.06.009}

\bibitem{arnold2006feec}
Arnold, D.N., Falk, R.S., Winther, R.: Finite element exterior calculus,
  homological techniques, and applications.
\newblock Acta Numerica \textbf{15}, 1--155 (2006).
\newblock \doi{10.1017/S0962492906210018}

\bibitem{balsara2004second}
Balsara, D.S.: Second-order-accurate schemes for magnetohydrodynamics with
  divergence-free reconstruction.
\newblock The Astrophysical Journal Supplement Series \textbf{151}(1), 149--184
  (2004).
\newblock \doi{10.1086/381377}

\bibitem{balsara1999staggered}
Balsara, D.S., Spicer, D.S.: A staggered mesh algorithm using high order
  godunov fluxes to ensure solenoidal magnetic fields in magnetohydrodynamic
  simulations.
\newblock Journal of Computational Physics \textbf{149}(2), 270--292 (1999).
\newblock \doi{10.1006/jcph.1998.6153}

\bibitem{barter2010shock}
Barter, G.E., Darmofal, D.L.: Shock capturing with {PDE}-based artificial
  viscosity for {DGFEM}: {Part I}. formulation.
\newblock Journal of Computational Physics \textbf{229}(5), 1810--1827 (2010).
\newblock \doi{10.1016/j.jcp.2009.11.010}

\bibitem{BRACKBILL1980426}
Brackbill, J.U., Barnes, D.C.: The effect of nonzero $\nabla \cdot \boldsymbol
  b$ on the numerical solution of the magnetohydrodynamic equations.
\newblock Journal of Computational Physics \textbf{35}(3), 426--430 (1980).
\newblock \doi{10.1016/0021-9991(80)90079-0}

\bibitem{chan2018discretely}
Chan, J.: On discretely entropy conservative and entropy stable discontinuous
  {Galerkin} methods.
\newblock Journal of Computational Physics \textbf{362}, 346--374 (2018).
\newblock \doi{10.1016/j.jcp.2018.02.033}

\bibitem{chandrashekar2013kinetic}
Chandrashekar, P.: Kinetic energy preserving and entropy stable finite volume
  schemes for compressible {Euler} and {Navier--Stokes} equations.
\newblock Communications in Computational Physics \textbf{14}(5), 1252--1286
  (2013).
\newblock \doi{10.4208/cicp.170712.010313a}

\bibitem{cohen1998mass}
Cohen, G., Monk, P.: Gauss point mass lumping schemes for {Maxwell}'s
  equations.
\newblock Numerical Methods for Partial Differential Equations \textbf{14}(1),
  63--88 (1998).
\newblock \doi{10.1002/(SICI)1098-2426(199801)14:1<63::AID-NUM4>3.0.CO;2-J}

\bibitem{dao2024structure}
Dao, T., Nazarov, M., Tomas, I.: A structure preserving numerical method for
  the ideal compressible {MHD} system.
\newblock Journal of Computational Physics \textbf{508}, 113009 (2024).
\newblock \doi{10.1016/j.jcp.2024.113009}

\bibitem{derigs2016mhdsolver}
Derigs, D., Winters, A.R., Gassner, G.J., Walch, S.: A novel high-order,
  entropy stable, {3D} {AMR} {MHD} solver with guaranteed positive pressure.
\newblock Journal of Computational Physics \textbf{317}, 223--256 (2016).
\newblock \doi{10.1016/j.jcp.2016.04.048}

\bibitem{dzanic2023mhdentropyfilter}
Dzanic, T., Witherden, F.D.: Positivity-preserving entropy filtering for ideal
  magnetohydrodynamics.
\newblock Computers \& Fluids \textbf{266}, 106056 (2023).
\newblock \doi{10.1016/j.compfluid.2023.106056}

\bibitem{evans1988ct}
Evans, C., Hawley, J.: Simulation of magnetohydrodynamic flows: A constrained
  transport method.
\newblock The Astrophysical Journal \textbf{332}, 659--677 (1988).
\newblock \doi{10.1086/166684}

\bibitem{FuLiu26}
Fu, G., Liu, J.G.: Entropy-stable and physical-constraint-preserving {DGSEM}
  for symmetry-reduced general-relativistic hydrodynamics on stationary
  spacetimes.
\newblock arXiv:2608.29229 (2026).
\newblock \urlprefix\url{https://arxiv.org/abs/2608.29229}

\bibitem{10.1007/s10915-018-0750-6}
Fu, P., Li, F., Xu, Y.: Globally divergence-free discontinuous galerkin methods
  for ideal magnetohydrodynamic equations.
\newblock Journal of Scientific Computing \textbf{77}(3), 1621--1659 (2018).
\newblock \doi{10.1007/s10915-018-0750-6}

\bibitem{fuchs2009splitting}
Fuchs, F., Mishra, S., Risebro, N.: Splitting based finite volume schemes for
  ideal {MHD} equations.
\newblock Journal of Computational Physics \textbf{228}(3), 641--660 (2009).
\newblock \doi{10.1016/j.jcp.2008.09.027}

\bibitem{gardiner2005unsplit}
Gardiner, T., Stone, J.: An unsplit {Godunov} method for ideal {MHD} via
  constrained transport.
\newblock Journal of Computational Physics \textbf{205}(2), 509--539 (2005).
\newblock \doi{10.1016/j.jcp.2004.11.016}

\bibitem{gassner2016split}
Gassner, G.J.: A skew-symmetric discontinuous {Galerkin} spectral element
  discretization and its relation to {SBP-SAT} finite difference methods.
\newblock SIAM Journal on Scientific Computing \textbf{35}(3), A1233--A1253
  (2013).
\newblock \doi{10.1137/120890144}

\bibitem{gassner2016splitform}
Gassner, G.J., Winters, A.R., Kopriva, D.A.: Split form nodal discontinuous
  {Galerkin} schemes with summation-by-parts property for the compressible
  {Euler} equations.
\newblock Journal of Computational Physics \textbf{327}, 39--66 (2016).
\newblock \doi{10.1016/j.jcp.2016.09.013}

\bibitem{goedbloed2019magnetohydrodynamics}
Goedbloed, J.P., Keppens, R., Poedts, S.: Magnetohydrodynamics of Laboratory
  and Astrophysical Plasmas.
\newblock Cambridge University Press, Cambridge (2019).
\newblock \doi{10.1017/9781316403679}

\bibitem{gottlieb2001strong}
Gottlieb, S., Shu, C.W., Tadmor, E.: Strong stability-preserving high-order
  time discretization methods.
\newblock SIAM Review \textbf{43}(1), 89--112 (2001).
\newblock \doi{10.1137/S003614450036757X}

\bibitem{guermond2016wavespeed}
Guermond, J.L., Popov, B.: Fast estimation from above of the maximum wave speed
  in the {Riemann} problem for the {Euler} equations.
\newblock Journal of Computational Physics \textbf{321}, 908--926 (2016).
\newblock \doi{10.1016/j.jcp.2016.05.054}

\bibitem{guillet2019highorder}
Guillet, T., Pakmor, R., Springel, V., Chandrashekar, P., Klingenberg, C.:
  High-order magnetohydrodynamics for astrophysics with an adaptive mesh
  refinement discontinuous {Galerkin} scheme.
\newblock Monthly Notices of the Royal Astronomical Society \textbf{485}(3),
  4209--4246 (2019).
\newblock \doi{10.1093/mnras/stz314}

\bibitem{hu2017stable}
Hu, K., Ma, Y., Xu, J.: Stable finite element methods preserving
  {$\nabla\cdot\boldsymbol{B}=0$} exactly for {MHD} models.
\newblock Numerische Mathematik \textbf{135}(2), 371--396 (2017).
\newblock \doi{10.1007/s00211-016-0803-4}

\bibitem{hu2019structure}
Hu, K., Xu, J.: Structure-preserving finite element methods for stationary
  {MHD} models.
\newblock Mathematics of Computation \textbf{88}(316), 553--581 (2019).
\newblock \doi{10.1090/mcom/3341}

\bibitem{jones1997mhd}
Jones, T.W., Gaalaas, J.B., Ryu, D., Frank, A.: The {MHD} {Kelvin--Helmholtz}
  instability. {II}. the roles of weak and oblique fields in planar flows.
\newblock The Astrophysical Journal \textbf{482}(1), 230--244 (1997).
\newblock \doi{10.1086/304145}

\bibitem{Li2005LocallyDD}
Li, F., Shu, C.W.: Locally divergence-free discontinuous galerkin methods for
  mhd equations.
\newblock Journal of Scientific Computing \textbf{22-23}(1--3), 413--442
  (2005).
\newblock \doi{10.1007/s10915-004-4146-4}

\bibitem{liu2025mhd}
Liu, M., Wu, K.: Structure-preserving oscillation-eliminating discontinuous
  {Galerkin} schemes for ideal {MHD} equations: Locally divergence-free and
  positivity-preserving.
\newblock Journal of Computational Physics \textbf{527}, 113795 (2025).
\newblock \doi{10.1016/j.jcp.2025.113795}

\bibitem{liu2026gdf}
Liu, M., Wu, K., Yuan, Y.: High-order positivity-preserving and globally
  divergence-free discontinuous {Galerkin} methods for ideal {MHD} with
  conditional energy conservation.
\newblock Journal of Computational Physics \textbf{562}, 115031 (2026).
\newblock \doi{10.1016/j.jcp.2026.115031}

\bibitem{liu2022essentially}
Liu, Y., Lu, J., Shu, C.W.: An essentially oscillation-free discontinuous
  {Galerkin} method for hyperbolic systems.
\newblock SIAM Journal on Scientific Computing \textbf{44}(1), A230--A259
  (2022).
\newblock \doi{10.1137/21M140835X}

\bibitem{liu2025mhdentropy}
Liu, Y., Lu, J., Shu, C.W.: An entropy stable essentially oscillation-free
  discontinuous {Galerkin} method for solving ideal magnetohydrodynamic
  equations.
\newblock Journal of Computational Physics \textbf{530}, 113911 (2025).
\newblock \doi{10.1016/j.jcp.2025.113911}

\bibitem{lu2021oscillation}
Lu, J., Liu, Y., Shu, C.W.: An oscillation-free discontinuous {Galerkin} method
  for scalar hyperbolic conservation laws.
\newblock SIAM Journal on Numerical Analysis \textbf{59}(3), 1299--1324 (2021).
\newblock \doi{10.1137/20M1354192}

\bibitem{nedelec1980}
N{\'e}d{\'e}lec, J.C.: Mixed finite elements in $\mathbb{R}^3$.
\newblock Numerische Mathematik \textbf{35}(3), 315--341 (1980).
\newblock \doi{10.1007/BF01396415}

\bibitem{orszag1979small}
Orszag, S.A., Tang, C.M.: Small-scale structure of two-dimensional
  magnetohydrodynamic turbulence.
\newblock Journal of Fluid Mechanics \textbf{90}(1), 129--143 (1979).
\newblock \doi{10.1017/S002211207900210X}

\bibitem{pang2025ct}
Pang, D., Wu, K.: Provably positivity-preserving constrained transport scheme
  for {2D} and {3D} ideal magnetohydrodynamics.
\newblock Journal of Computational Physics \textbf{541}, 114312 (2025).
\newblock \doi{10.1016/j.jcp.2025.114312}

\bibitem{peng2025oedg}
Peng, M., Sun, Z., Wu, K.: {OEDG}: Oscillation-eliminating discontinuous
  {Galerkin} method for hyperbolic conservation laws.
\newblock Mathematics of Computation \textbf{94}, 1147--1198 (2025).
\newblock \doi{10.1090/mcom/3998}

\bibitem{persson2006subcell}
Persson, P.O., Peraire, J.: Sub-cell shock capturing for discontinuous
  {Galerkin} methods.
\newblock In: 44th {AIAA} Aerospace Sciences Meeting and Exhibit, pp. AIAA
  Paper 2006--0112 (2006).
\newblock \doi{10.2514/6.2006-112}

\bibitem{Powell1997}
Powell, K.G.: An approximate riemann solver for magnetohydrodynamics.
\newblock In: M.Y. Hussaini, B.~van Leer, J.~Van~Rosendale (eds.) Upwind and
  High-Resolution Schemes, pp. 570--583. Springer Berlin Heidelberg, Berlin,
  Heidelberg (1997).
\newblock \doi{10.1007/978-3-642-60543-7_23}.
\newblock \urlprefix\url{https://doi.org/10.1007/978-3-642-60543-7_23}

\bibitem{raviart1977}
Raviart, P.A., Thomas, J.M.: A mixed finite element method for second order
  elliptic problems.
\newblock In: I.~Galligani, E.~Magenes (eds.) Mathematical Aspects of Finite
  Element Methods, \emph{Lecture Notes in Mathematics}, vol. 606, pp. 292--315.
  Springer, Berlin, Heidelberg (1977).
\newblock \doi{10.1007/BFb0064470}

\bibitem{seo2023howmhd}
Seo, J., Ryu, D.: {HOW-MHD}: A high-order {WENO}-based magnetohydrodynamic code
  with a high-order constrained transport algorithm for astrophysical
  applications.
\newblock The Astrophysical Journal \textbf{953}(1), 39 (2023).
\newblock \doi{10.3847/1538-4357/acdf4b}

\bibitem{stone2008athena}
Stone, J.M., Gardiner, T.A., Teuben, P., Hawley, J.F., Simon, J.B.: Athena: A
  new code for astrophysical {MHD}.
\newblock The Astrophysical Journal Supplement Series \textbf{178}(1), 137--177
  (2008).
\newblock \doi{10.1086/588755}

\bibitem{strang1968construction}
Strang, G.: On the construction and comparison of difference schemes.
\newblock {SIAM} Journal on Numerical Analysis \textbf{5}(3), 506--517 (1968).
\newblock \doi{10.1137/0705041}

\bibitem{tadmor1987numerical}
Tadmor, E.: The numerical viscosity of entropy stable schemes for systems of
  conservation laws. i.
\newblock Mathematics of Computation \textbf{49}(179), 91--103 (1987).
\newblock \doi{10.1090/S0025-5718-1987-0890255-3}

\bibitem{tadmor2003entropy}
Tadmor, E.: Entropy stability theory for difference approximations of nonlinear
  conservation laws and related time-dependent problems.
\newblock Acta Numerica \textbf{12}, 451--512 (2003).
\newblock \doi{10.1017/S0962492902000156}

\bibitem{toth2000divergence}
T{\'o}th, G.: The \(\nabla\cdot\boldsymbol b=0\) constraint in shock-capturing
  magnetohydrodynamics codes.
\newblock Journal of Computational Physics \textbf{161}(2), 605--652 (2000).
\newblock \doi{10.1006/jcph.2000.6519}

\bibitem{wimmer2024compatible}
Wimmer, G.A., Tang, X.Z.: Structure preserving transport stabilized compatible
  finite element methods for magnetohydrodynamics.
\newblock Journal of Computational Physics \textbf{501}, 112777 (2024).
\newblock \doi{10.1016/j.jcp.2024.112777}

\bibitem{winters2016mhdentropy}
Winters, A.R., Gassner, G.J.: Affordable, entropy conserving and entropy stable
  flux functions for the ideal {MHD} equations.
\newblock Journal of Computational Physics \textbf{304}, 72--108 (2016).
\newblock \doi{10.1016/j.jcp.2015.09.055}

\bibitem{wu2018positivity}
Wu, K.: Positivity-preserving analysis of numerical schemes for ideal
  magnetohydrodynamics.
\newblock SIAM Journal on Numerical Analysis \textbf{56}(4), 2124--2147 (2018).
\newblock \doi{10.1137/18M1168017}

\bibitem{wu2018provably}
Wu, K., Shu, C.W.: A provably positive discontinuous {Galerkin} method for
  multidimensional ideal magnetohydrodynamics.
\newblock SIAM Journal on Scientific Computing \textbf{40}(5), B1302--B1329
  (2018).
\newblock \doi{10.1137/18M1168042}

\bibitem{wu2026nodal}
Wu, Y., Shu, C.W.: A positivity preserving and entropy stable nodal
  discontinuous {Galerkin} scheme for ideal {MHD}.
\newblock arXiv:2604.23885 (2026).
\newblock \urlprefix\url{https://arxiv.org/abs/2604.23885}

\bibitem{YangFu26}
Yang, J., Fu, G.: An entropy-stable oscillation-eliminating {DGSEM} for the
  {Euler} equations on curvilinear meshes.
\newblock arXiv:2602.16732 (2026).
\newblock \urlprefix\url{https://arxiv.org/abs/2602.16732}

\bibitem{zhang2010positivity}
Zhang, X., Shu, C.W.: On positivity-preserving high order discontinuous
  {Galerkin} schemes for compressible {Euler} equations on rectangular meshes.
\newblock Journal of Computational Physics \textbf{229}(23), 8918--8934 (2010).
\newblock \doi{10.1016/j.jcp.2010.08.016}

\end{thebibliography}

\end{document}